\documentclass[11pt]{article}

\usepackage[T1]{fontenc}
\usepackage[latin1]{inputenc}

\usepackage{lmodern}

\usepackage{amsmath}

\usepackage{listings}
\usepackage{tabularx}

\usepackage{amsfonts}
\usepackage{amsmath}
\usepackage{amsthm}
\usepackage{amssymb}
\usepackage{graphicx}
\usepackage{float} 

\usepackage{bbm}
\renewcommand{\epsilon}{\varepsilon}

\usepackage[margin=12pt,font=footnotesize,labelfont=bf,
labelsep=endash]{caption}

\theoremstyle{plain} 
\newtheorem{thm}{Theorem}[section]

\newtheorem{lemma}[thm]{Lemma}
\newtheorem{prop}[thm]{Proposition}

\theoremstyle{plain} 
\newtheorem{defin}[thm]{Definition}

\newcommand{\Rs}{\mathbb{R}}

\newcommand{\Ns}{\mathbb{N}}

\newcommand{\Pb}{\mathbb{P}}
\newcommand{\E}{\mathbb{E}}

\newcommand{\mfinite}{{\cal M}^f_+}

\newcommand{\mprob}{{\cal M}_1}
\newcommand{\dkr}{{d_{\rm KR}}}
\newcommand{\eyenot}{{\cal I}_0}
\newcommand{\eyeone}{{\cal I}_1}

\newcommand{\eyep}{{I_{Poi}}}
\newcommand{\eyeq}{{J}}
\newcommand{\eyelamb}{{I}}
\newcommand{\eyecox}{{I}}
\newcommand{\eyecoxpp}{{\mathcal{I}}}
\newcommand{\eyex}{{H_x}}
\newcommand{\eyexn}{{H_{x_n}}}
\newcommand{\vareyeone}{\mathfrak{I}_1}
\newcommand{\vareyetwo}{\mathfrak{I}_2}
\newcommand{\eqdistrib}{\stackrel{\scriptstyle{\rm d}}{=}}

\usepackage{MnSymbol}

\begin{document}

\title{Functional Large Deviations for Cox Processes and $Cox/G/\infty$ Queues, with a Biological Application}
\author{Justin Dean, Ayalvadi Ganesh, Edward Crane} 

\maketitle

\section*{Abstract} 

We consider an infinite-server queue into which customers arrive according to a Cox process and have independent service times 
with a general distribution. We prove a functional large deviations principle for the equilibrium queue length process. The model is 
motivated by a linear feed-forward gene regulatory network, in which the rate of protein synthesis is modulated by the number 
of RNA molecules present in a cell. The system can be modelled as a non-standard tandem of infinite-server queues, in which 
the number of customers present in a queue modulates the arrival rate into the next queue in the tandem. We establish large 
deviation principles for this queueing system in the asymptotic regime in which the arrival process is sped up, while the service 
process is not scaled.  

\newpage

\section{Introduction} \label{sec:intro}

The work in this paper is motivated by the problem of modelling fluctuations in the number of protein molecules in a cell. 
The synthesis of proteins is catalysed by RNA molecules, which in turn are transcribed from DNA molecules. Both RNA and 
protein molecules degrade spontaneously after some random time. It is important for proper functioning of the cell that 
protein numbers are maintained within certain limits, and biologists are interested in understanding the regulatory mechanisms 
involved in controlling their fluctuations. Consequently, the problem of modelling stochastic fluctuations has attracted 
interest, and there has been considerable work on Markovian models of such systems; see, e.g., ~\cite{lestas08, paulsson05}. 
These models assume that each copy of a gene creates RNA molecules according to a Poisson process (while active), 
that each RNA molecule generates protein molecules according to a Poisson process, and that the lifetimes of RNA and 
protein molecules are exponentially distributed. The assumption of exponential lifetimes is biologically unrealistic; for example, 
inhomogeneities in the cellular environment could result in lifetimes that are mixtures of exponential distributions, or the 
denaturing of molecules could be a multistage process. 

Our approach relies on modelling the chemical kinetics using $\cdot/G/\infty$ queues rather than Markov processes, which 
correspond to $\cdot/M/\infty$ queues. Customer arrivals into the queue correspond to the synthesis of molecules of a 
specified type; after independent lifetimes with a general distribution, the molecules decay which equates to service (and 
departure) of the corresponding customers. For the problem described above, we have two such queues in series, one 
for RNA molecules and one for proteins. However, unlike in a tandem queueing network, where departures from one queue 
enter the next queue in series, here departures just leave the system; the way influence propagates is that the arrival rate 
into the protein queue is modulated by the occupancy of the preceding queue (here, RNA) in the series. We consider a very 
simple form of modulation, in which the arrival rate into a queue is proportional to the occupancy of the preceding queue, 
and the arrival process is conditionally Poisson given the occupancy. Thus, this results in a Cox process model for the arrivals 
into a queue, and the system is modelled as a series of $Cox/G/\infty$ queues interacting as described. 

We briefly recall the description of the queue length process in an $M/G/\infty$ queue with arrival rate $\lambda$ and  
service distribution $F$. The arrival process into this queue can be represented as an inhomogeneous Poisson process on 
$\Rs \times \Rs_+$ with intensity measure $\lambda \otimes F$. If a realisation of this point process has a point at 
$(t,y)$, it denotes that a customer arrives at time $t$ bringing a service requirement of $y$. The queue length at time $t$ 
is simply the total number of points of the Poisson process in the set 
$$
A_t= \{ (s,y): s\le t, y> t-s \},
$$
as a customer arriving at time $s$ will still be in the system at time $t$ if and only if its service requirement is greater than $t-s$. 
(We follow the convention of defining the queue length process to be right continuous.)  Likewise, the queue length process during 
a time interval $[s,t]$ can be described in terms of the empirical measure of the above Poisson process on the wedge-shaped set 
$$
A_{[s,t]} = \bigcup_{u \in [s,t]} A_u.
$$

In the problem we want to study, the intensity of the arrival process is modulated by the number of customers present 
in the previous queue. Hence, we need to model it as a Cox process and study the corresponding $Cox/G/\infty$ queue. 
As described above, this requires us to study the empirical measure of a Cox process on a subset of $\Rs^2$. We shall 
in fact study them in a more general setting of $\sigma$-compact Polish spaces, namely Polish spaces that can be 
covered by countably many compact subsets. Our goal is to obtain functional large deviation principles (FLDPs) for the 
corresponding queue length processes; we shall obtain these by contraction from LDPs for the empirical measure of the 
Cox process. We have not been able to drop the technical assumption of $\sigma$-compactness from our proof, but do 
not know if it is essential for the stated results.

In terms of the motivating application, biologists have been interested in understanding fluctuations in molecule numbers, 
both because large fluctuations can be deleterious, and because the statistics of fluctuations can shed light on underlying 
regulatory mechanisms. Most work to date has focused on second-order statistics, both the marginal variance, and 
auto-covariance and cross-covariance functions for protein and mRNA molecule counts~\cite{lestas08}. The scaling 
regime studied in this paper might be more relevant for understanding the rare but large fluctuations that are most 
harmful for the cell. Our methods could also provide the foundation for an analysis of regulatory mechanisms, which 
we do not consider in this paper. Finally, a functional LDP can be used to identify the most likely paths leading to rare 
events of interest, and thereby to design efficient simulation schemes via importance sampling for estimating these 
probabilities more accurately.

We present our model and main results in the next section, followed by the proofs in the final two sections.

\section{Model and Results} \label{sec:model}

We now set out our Cox process model. Let $(E,d)$ be a $\sigma$-compact Polish space, and let $\Lambda$ be a 
random finite Borel measure on $E$; in other words, $\Lambda$ is a random variable taking values in $\mfinite(E)$, 
the space of finite non-negative Borel measures on $E$. A Cox process $\Phi$ with stochastic intensity $\Lambda$ 
is a point process which is conditionally Poisson, with intensity measure $\lambda$ on the event that $\Lambda=\lambda$. 
Note that the point process $\Phi$ is almost surely finite. A realisation of $\Phi$ can be thought of as either a point set 
$\{ x_1,x_2,\ldots,x_k \}$, or as a counting measure $\sum_{i=1}^k \delta_{x_i}$, where $k$ is the (random) number 
of points in the realisation. We call the latter the empirical measure corresponding to the realisation of the point set, and 
note that it is also an element of $\mfinite(E)$. There are two topologies on $\mfinite(E)$ which will be of interest to us. 
We say that a sequence of measures $\mu_n \in \mfinite(E)$ converges to $\mu \in \mfinite(E)$ in the weak topology if 
$\int_E fd\mu_n$ converges to $\int_E fd\mu$ for all bounded continuous functions $f:E\to \Rs$; we say the measures 
converge in the vague topology if the integrals converge only for continuous functions with compact support (which are 
necessarily bounded). 

We now consider a sequence of Cox point processes $\Phi_n$, with corresponding stochastic intensities $\Lambda_n$. 
Our first contribution is a large deviation principle (LDP) for their scaled empirical measures:

\begin{thm} \label{thm:cox_ldp}
Suppose that $(\Lambda_n, n\in \Ns)$ is a sequence of random finite Borel measures on a $\sigma$-compact Polish space 
$(E,d)$, and that the sequence $\Lambda_n/n$ satisfies an LDP in $\mfinite(E)$ equipped with the weak topology, with 
good rate function $\vareyeone(\cdot)$. Let $\Phi_n$ be a Cox process with stochastic intensity $\Lambda_n$, i.e., a 
random counting measure on $E$ equipped with its Borel $\sigma$-algebra. Then the sequence of measures $\Phi_n/n$ 
satisfies an LDP in $\mfinite(E)$ equipped with the weak topology, with good rate function
$$
\vareyetwo(\mu) = \begin{cases}
\inf_{\lambda} \left\{\vareyeone(\lambda)+\lambda(E)\right\}, & \mbox{ if $\mu \equiv 0$,} \\
\inf_{\lambda} \left\{\vareyeone(\lambda) + \eyep(\mu(E),\lambda(E)) +\mu(E) H\bigl( \frac{\mu}{\mu(E)} \bigm| 
\frac{\lambda}{\lambda(E)} \bigr)\right\}, & \mbox{ if $\mu \not \equiv 0$,}
\end{cases}
$$
where $H$ and $\eyep$ are defined as follows:
\begin{eqnarray*} 
H(\beta|\alpha) &=&\left\{
	\begin{array}{ll}
		\int\log(d\beta/d\alpha)d\beta  & \mbox{if } \beta \ll \alpha \mbox{ and } \int|\log(d\beta/d\alpha)|d\beta < \infty \\
		+\infty & \mbox{otherwise, }
	\end{array}
\right. \\
\eyep(x,\alpha) &=& \begin{cases}
x \log \frac{x}{\alpha}-x+\alpha, & \mbox{ if } \alpha>0, \\
0, & \mbox{ if } \alpha=0, x=0, \\ 
+\infty, & \mbox{ if } \alpha=0, x>0.
\end{cases}
\end{eqnarray*}
The function $H(\beta|\alpha)$ is called the relative entropy or Kullback-Leibler divergence of $\beta$ with respect to $\alpha$.
\end{thm}

A slightly different version of this theorem, with only local finiteness of the measures $\Lambda_n$ assumed, has been established 
by Schreiber~\cite{schreiber03}, albeit in the vague rather than the weak topology; his result also requires a technical assumption 
about the measures $\Lambda_n/n$ dominating a fixed measure with full support on $E$, which we do not need. However, his result 
does not require that the space be $\sigma$-compact. The extension of the result to the weak topology is non-trivial, and relies on the 
finiteness assumption on the intensity measures. In addition, our proof techniques are very different. A functional LDP for rescaled 
Poisson random measures  is proved in \cite{florens98} using projective limits, and in \cite{leonard00} using Cram\'er's theorem and subadditivity arguments. 

The claim of Theorem~\ref{thm:cox_ldp} appears intuitive from the assumed LDP for the intensity measures $\Lambda_n/n$, 
the LDP for a Poisson random variable, and Sanov's theorem for the empirical distribution. However, a number of technical conditions need to be checked. Moreover, while these imply an LDP, goodness of the rate function is not immediate. We show this indirectly by establishing exponential tightness; this is the step where finiteness of the measures is crucial.

Next, we consider a sequence of stationary $Cox/G/\infty$ queues where the arrival processes are sped up 
by the index $n \in \Ns$, while the service process remains unchanged. More precisely, the service times are iid with some 
fixed distribution $F$ that does not depend on $n$, while the arrival process into the $n^{\rm th}$ queue is a Cox process with 
stochastic intensity (directing measure) $\Lambda_n$ on $\Rs$. We make the following assumptions.

\noindent\textbf{Assumptions}
\begin{enumerate}
\item[A1]
$(\Lambda_n, n\in \Ns)$ is a sequence of random $\sigma$-finite measures on $\Rs$, whose laws are translation invariant, 
such that $\E[\Lambda_n([a,b])] = n\lambda(b-a)$, for some fixed $\lambda>0$, and any compact interval $[a,b]\subset \Rs$. 
\item[A2]
For any interval $[a,b]$, the sequence $(\Lambda_n/n)|_{[a,b]}$ obeys an LDP on $\mfinite([a,b])$ equipped with the 
weak topology, with good rate function $\eyelamb_{[a,b]}$. 
\item[A3]
Define 
\begin{align*} \psi_n(\theta) = \log \E\left[ e^{\frac{\theta \Lambda_n([0,1])}{n}}\right].
\end{align*} 
There is a neighbourhood of $0$ on which 
$\psi_n(n\theta)/n$  is bounded, uniformly in $n$.
\item[A4] 
The mean service time, given by $\int_0^{\infty} xdF(x) = \int_0^{\infty} \overline{F}(x)dx$, is finite; here $\overline{F}=1-F$ 
denotes the complementary cumulative distribution function of the service time.
\end{enumerate}

Let $Q_n(t)$ denote the number of customers at time $t$ in the infinite-server queue with Cox process arrivals with intensity 
$\Lambda_n$ and iid service times with distribution $F$. Let $L_n$ denote the measure on $\Rs$ which is absolutely continuous 
with respect to Lebesgue measure, with density $Q_n(\cdot)$. Our second contribution in this paper is the following:

\begin{thm} \label{thm:qldp}
Consider a sequence of $Cox/G/\infty$ queues indexed by $n\in \Ns$, where the arrival process into the $n^{\rm th}$ 
queue is a Cox process with directing measure $\Lambda_n$, and service times are iid with common distribution $F$. 
Suppose the arrival and service processes satisfy Assumptions [A1]-[A4]. Let $Q_n(t)$ denote the number of 
customers in the $n^{\rm th}$ queue at time $t$, and let $L_n$ denote the random measure on $\Rs$ which is absolutely 
continuous with respect to Lebesgue measure and has density $Q_n(\cdot)$. Then the sequence of measures $L_n$ 
satisfies Assumptions [A1]-[A3]. In particular, for any compact interval $[a,b]\subset \Rs$, the measures $(L_n/n)|_{[a,b]}$ 
satisfy an LDP on $\mfinite([a,b])$ equipped with the weak topology, with a good rate function $\eyeq_{[a,b]}$. 
Moreover, the sequence of random variables $Q_n(0)/n$, satisfy an LDP with a good rate function $I_Q$.
\end{thm}

A fuller description of the rate functions $\eyeq_{[a,b]}$ and $I_Q$ is provided in the proof of this theorem, in 
Section~\ref{sec:qldp}. The theorem shows that the sequence of queue occupancy measures $L_n$ also satisfies the 
above assumptions and, in particular, that they satisfy an LDP. This implies that our analysis extends easily to an arbitrary 
number of $Cox/G/\infty$ queues in (non-standard) tandem, where the arrivals into each queue constitute a Cox process 
with directing measure given by the number in the previous queue. 

Can we prove an LDP, not just for the queue occupancy measures but for the queue lengths at a fixed time, say for the 
sequence of random variables $Q_n(0)/n = \hat\Phi_n(A_0)/n$? Unfortunately, the map $\mu \mapsto \mu(A_0)$ is not 
continuous in the weak topology, since the indicator function of the set $A_0$ is not a continuous function. Hence, our 
approach of invoking the Contraction Principle does not work. It might be possible to get around this, by sandwiching the 
indicator function of $A_0$ between bounded continuous functions which converge to it pointwise from below and above. 
We could then prove an LDP for the integral of the queue occupancy measure against these functions. If we could calculate 
the rate function explicitly, and show that it approaches the same limit for the functions approximating the indicator from 
above and from below, then that would prove the LDP for the marginal queue length distribution. But as these calculations 
are quite involve, and distract from the main motivation of the present work, we do not pursue them here. 

Next, we turn to the departure process from a $Cox/G/\infty$ queue. While it is not directly relevant to the model 
motivating this work, it is relevant to reaction networks in which the products of one reaction are reactants in the next, 
rather than catalysts as in our model. In that case, one would have a standard tandem of infinite-server queues, 
instead of the non-standard tandems that are the focus of this paper. In addition, the departure process is an object 
of interest in queueing theory. With these motivations, we now describe our results for the departure process.

Let $\Phi_n$ denote the Cox point process of arrivals into the $n^{\rm th}$ system as above, with directing measure 
$\Lambda_n$. Denote by $\hat\Phi_n$ the marked point process obtained by marking each arrival with its service time. 
Let $\Psi_n$ denote the point process of departures, which may be viewed as a random counting measure on $\Rs$. 
From the description of the $\cdot/G/\infty$ queue in terms of point processes given in the Introduction, we see that 
for any interval $[a,b]$, we have 
\begin{equation} \label{eq:dep_measure}
\Psi_n([a,b]) = \hat\Phi_n( {\rm cl}(A_{[a,b]} \backslash A_b)),
\end{equation}
since a customer departs during the interval $[a,b]$ only if it arrives at time $t\leq b$, briniging in an amount of work 
$x$ such that $a \leq t+x\leq b$; here ${\rm cl}(B)$ denotes the closure of a subset $B$ of $\Rs^2$. 
Our next result establishes an LDP for the empirical measures, $\Psi_n$, of the departures from the queue. Hence, the results 
extend easily to a (standard) tandem of such queues. 
\begin{thm} \label{thm:dep_ldp}
Let $\Phi_n$, $n\in \Ns$, be a sequence of Cox arrival processes satisfying Assumptions [A1]-[A3], and let $\hat\Phi_n$ 
be a Cox process obtained by marking the arrivals with iid service times drawn according to a distribution $F$ satisfying 
Assumption [A4]. Let $\Psi_n$ denote the corresponding departure process from an infinite-server queue, as defined 
precisely in (\ref{eq:dep_measure}). Then, $(\Psi_n, n\in \Ns)$ satisfies [A1]-[A3]; in particular, for any fixed compact 
interval $[a,b]$, $(\Psi_n/n)|_{[a,b]}$ obeys an LDP on $\mfinite([a,b])$ equipped with the weak topology, with a good 
rate function $K_{[a,b]}$.
\end{thm}

The $Cox/G/\infty$ model studied in this paper is an instance of a queue in a random environment. The first study of 
infinite-server queues in random environment was in~\cite{ocinneide86}: factorial moments in stationarity were derived 
for the $M/M/\infty$ queue in a Markovian environment, namely one in which the arrival and service rates are modulated 
by a finite state, irreducible, continuous time Markov chain. There has recently been extensive further study of this model, 
including moments for steady state and transient distributions, and large deviation and central limit asymptotics for the 
marginal distribution of the queue length; see~\cite{blom16} for a collation of the results. 
The Markovian assumption on the environment is relaxed in~\cite{jansen16}, where the background process modulating arrivals 
and services in an $M/M/\infty$ queue is just a general c\`adl\`ag stochastic process. An LDP is proved for the queue length at 
an arbitrary fixed time, $t$, whereas we establish a process level LDP, without assuming (conditionally) exponential service times. 
A special type of Cox background process is considered in ~\cite{heemskerk17}, which proves a functional CLT for the scaled 
queue length process. In all of these cases the queue length is viewed as a random c\`adl\`ag function, whereas we view it as 
living on a space of measures.

The proof of Theorem~\ref{thm:cox_ldp} is presented in Section~\ref{sec:cox_ldp}, and the proofs of Theorems 
\ref{thm:qldp} and \ref{thm:dep_ldp} in Section~\ref{sec:qldp}.

\section{Proof of Empirical Measure LDP} \label{sec:cox_ldp}

Our proof of Theorem~\ref{thm:cox_ldp}  relies on a theorem of Chaganty~\cite{chaganty97}, which essentially states that a sequence of probability measures on a product space satisfies an LDP if the corresponding sequences of marginal and conditional probability distributions do so, and certain additional technical conditions are satisfied. For completeness, we include below a statement of this theorem, together with an extension of Sanov's theorem by Baxter and Jain~\cite{baxter88} which is needed to check its conditions, and relevant definitions.

\begin{defin}\label{chagdef} \normalfont
Let $\left(\Omega_1,\mathcal{B}_1\right)$ and $\left(\Omega_2,\mathcal{B}_2\right)$ be two Polish spaces with their associated Borel $\sigma-$fields. Let  $\left\{\nu_n(\cdot,\cdot)\right\}$ be a sequence of transition functions on $\Omega_1\times\mathcal{B}_2$, i.e., $\nu_n(x_1,\cdot)$ is a probability measure on $\left(\Omega_2,\mathcal{B}_2\right)$ for each $x_1\in \Omega_1$ and $\nu_n(\cdot, B_2)$ is a measurable function on $\Omega_1$ for each $B_2 \in \mathcal{B}_2$. We say that the sequence of probability transition functions $\left\{\nu_n(x_1,\cdot),x_1\in\Omega_1\right\}$ satisfies the LDP continuously in $x_1$ with rate function $J(x_1,x_2)$, or simply the LDP continuity condition holds, if:
\begin{enumerate}
	\item For each $x_1\in\Omega_1$, $J(x_1,\cdot)$ is a good rate function on $\Omega_2$, i.e., it is non-negative, lower semicontinuous (l.s.c.), and has compact level sets.
	\item For any sequence $\left\{x_{1n}\right\}$ in $\Omega_1$ such that $x_{1n}\rightarrow x_1$, the sequence of measures $\left\{\nu_n(x_{1n},\cdot)\right\}$ on $\Omega_2$ obeys the LDP with rate function $J(x_1,\cdot)$.
	\item $J(x_1,x_2)$ is l.s.c. as a function of $(x_1,x_2)$.
\end{enumerate}
\end{defin}

\begin{thm}\label{chagthm}
(\cite[Theorem $2.3$]{chaganty97}) Let $\left(\Omega_1,\mathcal{B}_1\right)$, $\left(\Omega_2,\mathcal{B}_2\right)$ 
be two Polish spaces with their associated Borel $\sigma-$fields. Let $\left\{\mu_{1n}\right\}$ be a sequence of probability 
measures on $\left(\Omega_1,\mathcal{B}_1\right)$. Let $\left\{\nu_n(x_1,B_2)\right\}$ be a sequence of probability 
transition functions defined on $\Omega_1\times\mathcal{B}_2$. We define the joint distribution $\mu_n$ on the product space $\Omega_1 \times \Omega_2$, and the marginal distribution $\mu_{2n}$ on $\Omega_2$ by
$$
\mu_n(B_1\times B_2) = \int\limits_{B_1}\nu_n(x_1,B_2)d\mu_{1n}(x_1), \quad \mu_{2n}(B_2) = \mu_n(\Omega_1\times B_2).
$$ 
Suppose that the following two conditions are satisfied:
\begin{enumerate}
	\item $\left\{\mu_{1n}\right\}$ satisfies an LDP with good rate function $I_1(x_1)$.
	\item $\left\{\nu_n(\cdot,\cdot)\right\}$ satisfies the LDP continuity condition with a rate function $J(x_1,x_2)$.
\end{enumerate}
Then the sequence of joint distributions $\left\{\mu_n\right\}$ satisfies a weak LDP on the product space 
$\Omega_1\times\Omega_2$,  with rate function 
\begin{align*} I(x_1,x_2)=I_1(x_1)+J(x_1,x_2). 
\end{align*}
The sequence of marginal distributions $\mu_{2n}$ satisfies an LDP with rate function 
\begin{align*} I_2(x_2)=\inf\limits_{x_1\in\Omega_1}\left[I_1(x_1)+J(x_1,x_2)\right]. 
\end{align*}
Finally, $\left\{\mu_n\right\}$ satisfies the LDP if $I(x_1,x_2)$ is a good rate function.
\end{thm}

\noindent {\bf Remark.} Recall that a sequence of probability measures (or random variables) is said to satisfy a weak LDP if the large 
deviations upper bound holds for all compact sets, and to satisfy a (full) LDP if it holds for all closed sets. For both, the large 
deviations lower bound holds for all open sets.


\begin{thm}\label{baxterthm}
(\cite{baxter88}, Theorem $5$) Let $(S,d)$ be a Polish space. Let $\left\{\alpha_n\right\}$ be a sequence of probability measures on $(S,d)$ converging weakly to a probability measure $\alpha$. For each $n$, let $X^n_i$, $i\in \Ns$ be iid $S-$valued random variables with common distribution $\alpha_n$. Let $\mathcal{M}_1(S)$ denote the space of probability measures on $S$ and let $\overline{\mu}_n \in \mathcal{M}_1(S)$ denote the empirical distribution, $\left(\delta_{X^n_1}+...+\delta_{X^n_n}\right)/n$. Then $\left\{\overline{\mu}_n\right\}$ satisfies the LDP with good rate function $H(\cdot|\alpha)$, which was defined in the statement 
of Theorem~\ref{thm:cox_ldp}. 
\end{thm}

The proof of Theorem~\ref{thm:cox_ldp} proceeds through a sequence of lemmas. We begin with an elementary LDP 
for a sequence of Poisson random variables.

\begin{lemma} \label{poisson_ldp}
Let $N_n, n\in \Ns$ be a sequence of Poisson random variables with parameter $n\alpha_n$, and suppose that $\alpha_n$ 
tends to $\alpha \ge 0$. Then the sequence $N_n/n$ obeys an LDP in $\Rs_+$ with good rate function $\eyep(\cdot,\alpha)$ 
defined in the statement of Theorem~\ref{thm:cox_ldp}.
\end{lemma}

\begin{proof} 
We apply the G\"artner-Ellis theorem \cite[Theorem 2.3.6]{dembo98} to the sequence $N_n/n$. By direct calculation, 
$$
\frac{1}{n}\log\mathbb{E}\left[e^{n\theta\frac{N_n}{n}}\right]=\alpha_n \bigl(e^{\theta}-1 \bigr).
$$
This sequence of scaled log-moment generating functions converges pointwise to the limit $\alpha(e^{\theta}-1)$, which is 
finite and differentiable everywhere (hence also continuous, and essentially smooth). Hence, by the G\"artner-Ellis theorem, 
the sequence of random variables $N_n/n$ obeys an LDP with a rate function which is the convex conjugate of 
$\alpha(e^{\theta}-1)$. A straightforward calculation confirms that this is the function $\eyep(\cdot,\lambda)$ in the 
statement of the lemma, and that it is l.s.c. with compact level sets for each $\alpha$.
\end{proof}

The next two lemmas establish conditional LDPs for the scaled empirical measures of Poisson processes whose scaled intensities converge to a limit.

\begin{lemma}\label{condldp0} 
Let $\Phi_n, n\in \Ns$ be a sequence of  Poisson point processes with intensity measures $n\lambda_n \in \mfinite(E)$, and 
suppose that $\lambda_n$ converge  weakly in $\mfinite(E)$  to the zero measure. Then,  $\Phi_n/n, n\in \Ns$ satisfy the 
LDP in $\mfinite(E)$ equipped with the weak topology, with good rate function
$$
\eyenot(\mu) = \begin{cases}
0, & \mbox{ if } \mu \equiv 0, \\
+\infty, & \mbox{ otherwise.}
\end{cases}
$$
\end{lemma}

\begin{proof}
As the map $\mu \mapsto \mu(E)$ is weakly continuous (the indicator of $E$ is a bounded, continuous function), it follows that 
$\lambda_n(E)$ tends to $\lambda(E)=0$. Let $N_n=\Phi_n(E)$ denote the total number of points in the Poisson process $\Phi_n$. 
Then, $N_n$ is a Poisson random variable with parameter $n\lambda_n(E)$, and it follows from Lemma~\ref{poisson_ldp} that 
$(N_n/n, n\in \Ns)$ obey an LDP with good rate function
$$
\eyep(x,0) = \begin{cases}
0, & \mbox{if } x=0, \\
+\infty, &\mbox{if }x>0.
\end{cases}
$$

Let $F \subset \mfinite(E)$ be closed in the weak topology, and suppose that it does not contain the zero measure. Define
$$
x_F = \inf \{ \mu(E): \mu \in F \}.
$$
We claim that $x_F>0$. Indeed, if $x_F=0$, then we can find a sequence of measures $\mu_n \in F$ such that $\mu_n(E)$ 
tends to zero, i.e., $\int_E 1d\mu_n$ tends to zero. It follows that $\int_E fd\mu_n$ tends to zero for all bounded, measurable, 
non-negative functions $f$, and hence also for all bounded measurable functions. Hence, the sequence $\mu_n$ converges 
weakly to the zero measure, contradicting the assumption that $0\notin F$ and $F$ is closed.

We now have the large deviations upper bound for $F$:
\begin{eqnarray*}
\limsup_{n\to \infty} \frac{1}{n}\log \Pb \Bigl( \frac{\Phi_n}{n} \in F \Bigr) 
&\le& \limsup_{n\to \infty} \frac{1}{n}\log \Pb \Bigl( \frac{\Phi_n(E)}{n} \ge x_F \Bigr) \\
&=& \limsup_{n\to \infty} \frac{1}{n}\log \Pb \Bigl( \frac{N_n}{n} \ge x_F \Bigr) \; = \; -\infty, 
\end{eqnarray*}
where we have used the LDP for $N_n/n$ with rate function $\eyep(\cdot,0)$ and the fact that $x_F>0$ to obtain the last equality.

The large deviations lower bound is trivial for open sets $G$ not containing the zero measure, as the infimum of the rate function 
is infinite on such sets. Now, for $G$ containing the zero measure, we have
\begin{eqnarray*}
\liminf_{n\to \infty} \frac{1}{n}\log \Pb \Bigl( \frac{\Phi_n}{n} \in G \Bigr) 
&\ge& \liminf_{n\to \infty} \frac{1}{n}\log \Pb \Bigl( \frac{\Phi_n}{n} \equiv 0 \Bigr) 
\; = \; \liminf_{n\to \infty} \frac{1}{n}\log \Pb (N_n=0) \\
&=& \liminf_{n\to \infty} (-\lambda_n(E)) \; = \; -\lambda(E) = 0,
\end{eqnarray*}
as $N_n\sim Poi(n\lambda_n(E))$. This completes the proof of the lemma. 
\end{proof}


\begin{lemma}\label{condldp1} 
Let $\Phi_n, n\in \Ns$ be a sequence of  Poisson point processes with intensity measures $n\lambda_n$, and suppose that 
the sequence $\lambda_n$ converges in the weak topology on $\mfinite(E)$  to $\lambda \not\equiv 0$. Then,  
$\Phi_n/n, n\in \Ns$ satisfy the LDP in $\mfinite(E)$ equipped with the weak topology, with good rate function
$$
\eyeone(\mu) = \begin{cases}
\eyep(\mu(E),\lambda(E)) + \mu(E) H \Bigl( \frac{\mu}{\mu(E)} \Bigm| \frac{\lambda}{\lambda(E)} \Bigr), 
& \mbox{ if } \mu \not\equiv 0, \\
\eyep(0,\lambda(E)), & \mbox{ if } \mu \equiv 0.
\end{cases}
$$
Here, $\eyep(\cdot,\cdot)$ and $H(\cdot| \cdot)$ are as defined in Lemma~\ref{poisson_ldp} and Theorem~\ref{baxterthm} respectively.
\end{lemma}

\begin{proof} 
We will prove the lemma by first establishing an LDP for the sequence $N_n/n$, then verifying that conditional on this, $\Phi_n/n$ 
satisfies the LDP continuously, and invoking Theorem~\ref{chagthm}.

The LDP for $N_n/n$, with rate function $\eyep(\cdot,\lambda(E))$, is immediate from Lemma~\ref{poisson_ldp} since $\lambda_n(E)$ 
tends to $\lambda(E)$. We now prove an LDP for $\Phi_n/n$, conditional on $N_n/n$. Fix a sequence $N_n$ such that $N_n/n \to x\ge 0$. If $x=0$, then the proof follows that 
of Lemma~\ref{condldp0}, and yields $\eyenot$ as the rate function.

It remains to consider $x>0$. We can write 
\begin{align*} \Phi_n = \delta_{X^n_1}+\delta_{X^n_2}+\ldots+\delta_{X^n_{N_n}},
\end{align*} 
where the $X^n_i$ are iid, with law $\frac{\lambda_n}{\lambda_n(E)}$. Note that the probability law of $X^n_i$ is well-defined for all $n$ 
sufficiently large, as $\lambda_n(E)$ tends to $\lambda(E)>0$. Define
$$
\hat \Phi_n = \delta_{X^n_1}+\delta_{X^n_2}+\ldots+\delta_{X^n_{\lfloor nx \rfloor}},
$$
where the dependence of $\hat \Phi_n$ on $x$ has been suppressed in the notation. We claim that the sequences 
$\Phi_n/n$ and $\hat \Phi_n/n$ are exponentially equivalent (see~\cite[Definition 4.2.10]{dembo98}). To see this, we 
use the fact that the weak topology on $\mfinite(E)$ can be metrised, for instance by the Kantorovich-Rubinstein metric, 
$$
\dkr(\mu,\nu) = \sup_{f\in {\rm Lip}(1), \| f \|_{\infty}\le 1} \; \int_E fd\mu - \int_E fd\nu.
$$
It is easy to see that, for all bounded measurable $f$,
$$
\Bigm| \int_E f d\Phi_n - \int_E fd\hat \Phi_n \Bigm| \; \leq \; \| f \|_{\infty} \; \bigm| N_n-\lfloor nx \rfloor \bigm|,
$$
and so, $\dkr(\Phi_n/n, \hat \Phi_n/n)$ tends to zero deterministically, as $N_n/n$ tends to $x$ deterministically. This 
establishes the exponential equivalence of the two sequences.

Now, we have from Theorem~\ref{baxterthm} and the observation that $\lambda_n(\cdot)/\lambda_n(E)$ converges 
weakly to $\lambda(\cdot)/\lambda(E)$, that $(\hat \Phi_n/\lfloor nx \rfloor, \lfloor nx \rfloor \in \Ns)$ obey an LDP in 
$\mprob(E)$ with good rate function $H \bigl(\cdot \bigm| \frac{\lambda}{\lambda(E)} \bigr)$, and hence also in 
$\mfinite(E)$ with rate function which is the same on $\mprob(E)$, and infinite outside it. It follows that $(\hat \Phi_n/n, n \in \Ns)$ 
obey an LDP in $\mfinite(E)$ with rate function
\begin{equation} \label{condldpx_rate}
\eyex(\mu) = \begin{cases}
xH \Bigl( \frac{\mu}{x} \Bigm| \frac{\lambda}{\lambda(E)} \Bigr), &\mbox{ if } \frac{\mu}{x} \in \mprob(E), \\
+\infty, &\mbox{ otherwise.}
\end{cases}
\end{equation}
Finally, by~\cite[Theorem 4.2.13]{dembo98}, $(\Phi_n/n, n \in \Ns)$ obey an LDP in $\mfinite(E)$ with the same rate function 
$\eyex$, as they are exponentially equivalent to $\hat \Phi_n/n$.

Having established conditional LDPs for $\Phi_n/n$, conditional on $N_n/n$ tending to $x$, we now need to check the LDP 
continuity conditions in Definition~\ref{chagdef} with $\Omega_1=\Rs_+$ and $\Omega_2=\mfinite(E)$, and transition 
function $\nu_n(x,\cdot)$ defined as the law of $\Phi_n$ conditional on $N_n=\lfloor nx \rfloor$. We defne the function 
$$
J(x,\mu) = \begin{cases}
\eyenot(\mu), &\mbox{ if } x=0, \\
\eyex(\mu) &\mbox{ if } x>0,
\end{cases}
$$
where $\eyenot$ is defined in Lemma~\ref{condldp0} and $\eyex$ in \eqref{condldpx_rate}. Note that $J$ is non-negative 
as $\eyenot$ and $\{ \eyex, x\ge 0 \}$ are all non-negative.

The first condition in Definition~\ref{chagdef} holds trivially if $x=0$, as all level sets are singletons comprised of the zero measure; 
if $x>0$, the condition follows from the goodness of the relative entropy function, which is well known from Sanov's theorem (see, 
e.g.,~\cite[Theorem 6.2.10]{dembo98}). In a bit more detail, given $\alpha>0$, the level set 
\begin{align*} L_{\alpha} = \left\{\mu\in \mprob(E): H\left(\mu \left|\frac{\lambda}{\lambda(E)}\right.\right) \le \frac{\alpha}{x}\right\}
\end{align*} 
is compact in $\mprob(E)$ equipped with the weak topology; hence, so is its image under the continuous map $\mu \mapsto x\mu$ 
from $\mprob(E)$ to $\mfinite(E)$.

The second condition in Definition~\ref{chagdef} is precisely the content of the conditional LDPs that we just obtained. That leaves us 
to check the third condition, which is that $J(x,\mu)$ is l.s.c. in $(x,\mu)$. As $\Rs_+ \times \mfinite(E)$ is a metric space, we 
can check this along sequences. Consider a sequence $(x_n,\mu_n)$ converging to $(x,\mu)$. If $(x,\mu)=(0,0)$, then $J(x,\mu)=0$, 
which is no bigger than $\liminf J(x_n,\mu_n)$. If $x=0$ and $\mu \not\equiv 0$, then $\mu(E)>0$ and so, for all $n$ sufficiently large, 
$x_n < \mu_n(E)$; consequently, $\mu_n/x_n$ is not a probability measure, and $J(x_n,\mu_n)=+\infty$. The same reasoning applies 
if $x>0$ and $\mu/x \notin \mprob(E)$. Finally, suppose $x>0$ and $\mu/x \in \mprob(E)$, so that $\mu_n/x_n$ converges weakly 
to $\mu/x$ in $\mfinite(E)$. We may restrict attention to the subsequence of $\Ns$ for which $\mu_n/x_n$ are probability measures, 
as $J(x_n,\mu_n)=+\infty$ otherwise. Along this subsequence, the desired inequality $\liminf \eyexn(\mu_n) \ge \eyex(\mu)$ 
follows from the lower semicontinuity of $H$, the relative entropy function.

We are now in a position to invoke Theorem~\ref{chagthm}, with $\Omega_1=\Rs_+$ and $\Omega_2=\mfinite(E)$. The second 
condition in the theorem is a conditional LDP for $\Phi_n/n$ given that $N_n/n$ tends to $x$, which we have just verified. The first 
condition is an LDP for $N_n/n$, which was proved in Lemma~\ref{poisson_ldp}. Hence, the conclusion of Theorem~\ref{chagthm} 
holds, i.e., we have an LDP for $\Phi_n/n$ with rate function
$$
I_2(\mu) = \inf_{x\in \Rs_+} \left\{\eyep(x,\lambda(E))+ J(x,\mu)\right\}.
$$
As $J(x,\mu)=+\infty$ unless $x=\mu(E)$, it is clear that the infimum is attained at $x=\mu(E)$, and we have 
\begin{align*} I_2(\mu)=\eyep(\mu(E),\lambda(E))+J(\mu(E),\mu).
\end{align*} 
This coincides with the rate function in the statement of the lemma, and concludes its proof. 
\end{proof}

We now have all the ingredients required to complete the proof of Theorem~\ref{thm:cox_ldp}.

\noindent{\bf Proof of Theorem~\ref{thm:cox_ldp}.}
We invoke Theorem~\ref{chagthm} with $\Omega_1$ and $\Omega_2$ both being the space of finite non-negative measures on $E$, 
equipped with the weak topology and the corresponding Borel $\sigma$-algebra. The sequence $\mu_{1n}$ will denote the laws 
of the directing (intensity) measures $\Lambda_n$, and the probability transition functions $\nu_n(\lambda,\cdot)$ will denote the 
law of the scaled Poisson random measures $\Phi_n/n$, where $\Phi_n$ has intensity $n\lambda$. We now check the assumptions 
of the theorem.

The first condition in Theorem~\ref{chagthm} is an LDP for $(\Lambda_n/n, n\in \Ns)$ with a good rate function, which holds by 
assumption. To check the second condition in Theorem~\ref{chagthm}, define
$$
J(\lambda,\mu) = \begin{cases}
\eyenot(\mu), &\mbox{ if } \lambda \equiv 0, \\
\eyeone(\mu), &\mbox{ otherwise,}
\end{cases}
$$
where $\eyenot$ and $\eyeone$ are as defined in Lemmas~\ref{condldp0} and \ref{condldp1}. We need to check that 
the conditions in Definition~\ref{chagdef} are satisfed. The first condition is satisfied as $\eyenot$ and $\eyeone$ are both 
good rate functions, as shown in Lemmas~\ref{condldp0} and \ref{condldp1}. The second condition is the content of the 
conditional LDPs established in these lemmas. That leaves us to check the third condition, that $J(\cdot,\cdot)$ is l.s.c..  
As the weak topology on $\mfinite(E)$ is metrisable, so is the product topology on $\mfinite(E)\times \mfinite(E)$, and 
we can check lower semicontinuity along sequences. Consider a sequence $(\lambda_n,\mu_n)$ converging to 
$(\lambda,\mu)$, i.e., $\lambda_n$ converges weakly to $\lambda$, and $\mu_n$ to $\mu$. We distinguish four cases:
\begin{enumerate}
\item
If $\lambda \equiv 0$ and $\mu \equiv 0$, then $J(\lambda,\mu)=\eyenot(\mu)=0$, which is no bigger than the limit infimum 
of a non-negative sequence.
\item
If $\lambda \equiv 0$ and $\mu \not \equiv 0$, then $J(\lambda,\mu)=\eyenot(\mu)=+\infty$. But note that 
$\lambda_n(E) \to \lambda(E)=0$ and $\mu_n(E) \to \mu(E)>0$, and so $\eyep(\mu_n(E),\lambda_n(E)) \to +\infty$. 
As 
\begin{align*} J(\lambda_n,\mu_n) = \eyeone(\mu_n) \ge \eyep(\mu_n(E),\lambda_n(E)),
\end{align*} 
we see that $J(\lambda_n,\mu_n)$ also tends to infinity.
\item
If $\lambda \not\equiv 0$ and $\mu \equiv 0$, then $J(\lambda,\mu)=\eyeone(\mu)=\eyep(0,\lambda(E))$. On the other hand, 
$J(\lambda_n,\mu_n) \ge \eyep(\mu_n(E),\lambda_n(E))$, which tends to $\eyep(0,\lambda(E))$ as $n$ tends to infinity, as 
$\eyep$ is continuous.
\item
Finally, suppose that $\lambda \not\equiv 0$ and $\mu \not \equiv 0$. In this case, for all $n$ sufficiently large, both 
$\lambda_n$ and $\mu_n$ are non-zero measures, and we have $J(\lambda_n,\mu_n)=\eyeone(\mu_n)$. As $\lambda_n(E)$ 
and $\mu_n(E)$ converge to $\lambda(E)$ and $\mu(E)$ respectively, it is easy to see that $\eyep(\mu_n(E),\lambda_n(E))$ 
tends to $\eyep(\mu(E),\lambda(E))$. Hence, to verify lower semicontinuity, it suffices to show that $H(\beta | \alpha)$ is 
jointly l.s.c. in its arguments. Recall the Donsker-Varadhan variational formula for the relative entropy (see, 
e.g.,~\cite[Sec. C.2]{dupuis97}):
\begin{eqnarray*} 
H(\beta|\alpha)=\sup_{g\in C_b(E)}\left\{\int\limits_E g d\beta- \log\int\limits_E e^{g}d\alpha \right\},
\end{eqnarray*}
where $C_b(E)$ denotes the set of bounded continuous functions on $E$. But if $g\in C_b(E)$, so is $e^g$, and the map 
\begin{align*} (\alpha,\beta)\longmapsto \int\limits_E gd\beta -\log\int\limits_E e^{g} d\alpha
\end{align*}
is continuous. Consequently, $H(\beta | \alpha)$, being the supremum of continuous functions of $(\alpha,\beta)$, is l.s.c..
\end{enumerate}

Thus, we have checked all the conditions of Theorem~\ref{chagthm}. Hence, the conclusion of the theorem holds, and yields that 
$(\Phi_n/n, n\in \Ns)$ obey an LDP on $\mfinite(E)$, with rate function
$$
\vareyetwo(\mu) = \inf_{\lambda \in \mfinite(E)} \left\{\vareyeone(\lambda) + J(\lambda, \mu)\right\},
$$
where $J(\lambda,\mu)$ equals $\eyenot(\mu)$ if $\lambda\equiv 0$ and $\eyeone(\mu)$ otherwise, and $\eyenot$ and 
$\eyeone$ are defined in Lemmas~\ref{condldp0} and \ref{condldp1} respectively. Using those definitions, we can write the 
rate function more explicitly as follows:
$$
\vareyetwo(\mu) = \begin{cases}
\inf_{\lambda} \left\{\vareyeone(\lambda)+\lambda(E)\right\}, & \mbox{ if $\mu \equiv 0$,} \\
\inf_{\lambda} \left\{\vareyeone(\lambda) + \eyep(\mu(E),\lambda(E)) +\mu(E) H\bigl( \frac{\mu}{\mu(E)} \bigm| 
\frac{\lambda}{\lambda(E)} \bigr)\right\}, & \mbox{ if $\mu \not \equiv 0$,}
\end{cases}
$$
where the infimum is taken over all finite Borel measures $\lambda$ on $E$. The expression above coincides with that in 
the statement of the theorem.

It remains only to check that the rate function $\vareyetwo$ is good. This is a consequence of Lemma~\ref{lem:exp_tight} 
below, which establishes the exponential tightness of the scaled empirical measures $\Phi_n/n$, and~\cite[Lemma 1.2.18]{dembo98}. 
This completes the proof of Theorem~\ref{thm:cox_ldp}. \hfill $\Box$

We first state a proposition which provides an explicit construction of compact subsets of $\mfinite(E)$, and which we will need 
for the proof of Lemma~\ref{lem:exp_tight}. The proof of the proposition is deferred until after the lemma, and is where the 
assumption of $\sigma$-compactness of $E$ is required.

\begin{prop} \label{compact_set_measures}
Let $K_1 \subseteq K_2 \subseteq \ldots$ be a nested sequence of compact subsets of $E$, whose union is equal to $E$; 
such a sequence exists by the assumption that $E$ is $\sigma$-compact. Let $\epsilon_0 \ge \epsilon_1 \ge \ldots$ be a 
sequence of real numbers decreasing to zero. Define $K_0$ to be the empty set. Then, the set 
$$
L_{(K_n,\epsilon_n)} = \bigl\{ \mu \in \mfinite(E): \mu(K_n^c) \le \epsilon_n \; \forall \; n\in \Ns \bigr\},
$$
is compact in the weak topology on $\mfinite(E)$. Moreover, if $\mathcal{K}$ is any compact subset of $\mfinite(E)$, 
and $\epsilon_n, n\in \Ns_+$ any sequence decreasing to 0, then there exist $\epsilon_0>0$ and compact 
$K_1 \subseteq K_2 \subseteq \ldots \subseteq E$ such that $\mathcal{K} \subseteq L_{(K_n,\epsilon_n)}$.
\end{prop}

\begin{lemma} \label{lem:exp_tight}
Suppose that $(\Lambda_n, n\in \Ns)$ is a sequence of random finite Borel measures on a Polish space $(E,d)$, which 
satisfy the assumptions of Theorem~\ref{thm:cox_ldp}. Let $(\Phi_n, n\in \Ns)$ be a sequence of Cox point processes 
on $E$, with stochastic intensities $\Lambda_n$. Then, the sequence of random measures $\Phi_n/n$ is exponentially 
tight in $\mfinite(E)$ equipped with the weak topology.
\end{lemma}

\begin{proof} 
We have to show that for every $\alpha<\infty$, there is a compact $\mathcal{K}_{\alpha} \subseteq \mfinite(E)$ such that 
\begin{equation} \label{eq:exptight_bound}
\limsup_{n\rightarrow\infty}\frac{1}{n}\log\ \Pb \Bigl( \frac{\Phi_n}{n} \in \mathcal{K}_{\alpha}^c \Bigr) <-\alpha.
\end{equation}

By the assumptions of Theorem~\ref{thm:cox_ldp}, the sequence $\Lambda_n/n$ satisfies an LDP in $\mfinite(E)$, 
with \emph{good} rate function $\vareyeone$. Hence, the sequence is exponentially tight, i.e., there is a compact set 
$\mathcal{\hat K}_{\alpha}\subseteq \mfinite(E)$ such that 
\begin{equation} \label{exptight_intensity}
\limsup_{n\rightarrow\infty} \frac{1}{n} \log \Pb \Bigl( \frac{\Lambda_n}{n}\notin \mathcal{\hat K}_{\alpha} \Bigr) <-\alpha.
\end{equation} 
By Proposition~\ref{compact_set_measures}, $\mathcal{\hat K}_{\alpha}$ is contained in a compact set of the form 
$L_{(K_n,\epsilon_n)}$, where $\epsilon_n, n\ge 1$ can be chosen to decrease to zero arbitrarily. We will show that, 
for a suitably chosen sequence $\delta_n \downarrow 0$, the set $L_{(K_n,\delta_n)}$ satisfies the upper bound in 
(\ref{eq:exptight_bound}).

Observe that
\begin{eqnarray}  
\Pb \Bigl( \frac{\Phi_n}{n}\notin L_{(K_i,\delta_i)} \Bigr) 
& \leq & \Pb \Bigl(\frac{\Phi_n}{n}\notin L_{(K_i,\delta_i)} \Bigm| \frac{\Lambda_n}{n} \in 
L_{(K_i,\epsilon_i)} \Bigr) +\Pb \Bigl( \frac{\Lambda_n}{n} \notin L_{(K_i,\epsilon_i)} \Bigr) \nonumber \\
& \leq & \Pb \Bigl( \frac{\Phi_n}{n}\notin L_{(K_i,\delta_i)} \Bigm| \frac{\Lambda_n}{n}\in L_{(K_i,\epsilon_i)} 
\Bigr) + \Pb \Bigl( \frac{\Lambda_n}{n} \notin \mathcal{\hat K}_{\alpha} \Bigr). \label{exptight_decomp1}
\end{eqnarray}

Now, conditional on $\Lambda_n$, $\Phi_n$ is a Poisson point process, and $\Phi_n(K_i^c)$ is a Poisson random 
variable with mean $\Lambda_n(K_i^c)$. Thus, conditional on $\Lambda_n/n \in L_{(K_i, \epsilon_i)}$, the random 
variable $\Phi_n(K_i^c)$ is stochastically dominated by a Poisson random variable with mean $n\epsilon_i$, for each 
$i\in \Ns$. Also, the event $\{ \Phi_n/n \notin L_{(K_i,\delta_i)} \}$ is the union of the events 
$\{ \Phi_n(K_i^c) > n\delta_i \}$ over $i\in \Ns$. Define $m_n = \sup \{ i: n\delta_i > 1 \}$. Since $\Phi_n$ is a 
counting measure, the event $\{ \Phi_n(K_i^c) > n\delta_i \}$ coincides with $\{ \Phi_n(K_i^c) \geq 1 \}$ for $i>m_n$.
Hence, we obtain using the union bound that
\begin{eqnarray} 
&& \Pb \Bigl( \frac{\Phi_n}{n}\notin L_{(K_i, \delta_i)} \Bigm| \frac{\Lambda_n}{n} \in L_{(K_i,\epsilon_i)} \Bigr)
\; \leq \; \sum_{i=0}^{\infty} \Pb \bigl( Poi ( n\epsilon_i )>n\delta_i \bigr) \nonumber \\
&& = \;  \sum_{i=0}^{m_n} \Pb \bigl( Poi ( n\epsilon_i )>n\delta_i \bigr) \; +
\sum_{i=m_n+1}^{\infty}\Pb \bigl( Poi ( n\epsilon_i ) \ge 1 \bigr). \label{exptight_decomp2}
\end{eqnarray}

Without loss of generality, we can take $\epsilon_0 \ge 1$. Take $\epsilon_i = e^{-i}$ and $\delta_i = \kappa/i$ 
for $i\ge 1$, for a constant $\kappa$ to be determined, depending on $\alpha$. Take $\delta_0 = \kappa \epsilon_0$. 
Then $m_n=\lfloor \kappa n \rfloor$, and we obtain using Markov's inequality that
\begin{equation} \label{tailsum_bound}
\sum_{i=m_n+1}^{\infty}\Pb \bigl( Poi ( n\epsilon_i ) \ge 1 \bigr) \le \sum_{i=\lceil \kappa n \rceil}^{\infty} ne^{-i} 
\le \frac{n e^{-\kappa n}}{1-e^{-1}}.
\end{equation}
We also have the large deviations (Chernoff) bound for a Poisson random variable that, for $\mu>\lambda$,
$$
\Pb \bigl( Poi(\lambda)>\mu ) \le \exp \Bigl( -\mu \log \frac{\mu}{\lambda}+\mu-\lambda \Bigr),
$$
from which it follows that
$$
\Pb \bigl( Poi ( n\epsilon_i )>n\delta_i \bigr) \le \begin{cases}
\exp(-n\epsilon_0 (\kappa \log \kappa-\kappa+1)), & i=0, \\
\exp \bigl( -n \kappa \frac{\log \kappa+i-1-\log i}{i} \bigr), & i\ge 1.
\end{cases}
$$
Now, $\epsilon_0 \ge 1$ by assumption and, if $\kappa$ is chosen sufficiently large, then it is easy to verify that
$(\log \kappa+i-1-\log i)/i$ is bigger than $1/2$ for all $i\ge 1$. Hence, we obtain that
\begin{equation} \label{bodysum_bound}
\sum_{i=0}^{m_n} \Pb \bigl( Poi ( n\epsilon_i )>n\delta_i \bigr) \le e^{-n (\kappa \log \kappa-\kappa+1)} 
 + \kappa n e^{-\kappa n/2},
\end{equation}
as $m_n=\lfloor \kappa n \rfloor$. Substituting (\ref{tailsum_bound}) and (\ref{bodysum_bound}) in 
(\ref{exptight_decomp2}), we get 
$$
\Pb \Bigl( \frac{\Phi_n}{n}\notin L_{(K_i, \delta_i)} \Bigm| \frac{\Lambda_n}{n} \in L_{(K_i,\epsilon_i)} \Bigr) 
\le \frac{n e^{-\kappa n}}{1-e^{-1}} + e^{-n (\kappa \log \kappa-\kappa+1)} + \kappa n e^{-\kappa n/2}.
$$ 
It is clear from this that we can choose $\kappa$ sufficiently large to ensure that
\begin{equation} \label{exptight_conditional}
\limsup_{n\to \infty} \frac{1}{n} \log \Pb \Bigl( \frac{\Phi_n}{n}\notin L_{(K_i, \delta_i)} \Bigm| 
\frac{\Lambda_n}{n} \in L_{(K_i,\epsilon_i)} \Bigr)  \le -\alpha.
\end{equation}
Finally, combining (\ref{exptight_intensity}), (\ref{exptight_decomp1}) and (\ref{exptight_conditional}), we conclude that 
$$
\limsup_{n\to \infty} \frac{1}{n} \log \Pb \Bigl( \frac{\Phi_n}{n}\notin L_{(K_i, \delta_i)} \Bigr) \le -\alpha.
$$
This concludes the proof of the lemma. 
\end{proof}

\noindent{\bf Proof of Proposition~\ref{compact_set_measures}.}
The weak topology on the space of finite measures on a Polish space is metrisable~\cite{varadarajan58}), and so it 
suffices to check sequential compactness. Let $(\mu_n, n\in \Ns)$ be a sequence of finite measures on $E$ satisfying 
the assumptions of the proposition with respect to a nested sequence of compact sets $K_n$ whose union is equal to $E$, 
and a sequence $\epsilon_n$ decreasing to zero. In particular, the measures are bounded; $\mu_n(E)\leq \epsilon_0$ 
for all $n\in \Ns$. We want to show that $(\mu_n, n\in \Ns)$ contains a convergent subsequence.

Recall that the space of subprobability measures on a compact set $K$ is compact in the weak topology; this follows 
from the Banach-Alaoglu theorem applied to the unit ball in the space of finite signed measures on $K$, which the 
Riesz representation theorem identifies with the dual of the Banach space $C(K)$ of continuous functions 
on $K$ equipped with the supremum norm. Hence, by Tychonoff's theorem, so is the space of finite measures on $K$ 
bounded by an arbitrary constant $\epsilon_0$.

Thus, the measures $\mu_n$ restricted to $K_1$ all lie within a compact set; hence, there is a subsequence 
$\mu_{11}, \mu_{12}, \ldots$, whose restriction to $K_1$ converges weakly to some $\tilde \mu_1 \in \mfinite(K_1)$.
Similarly, the restriction of this subsequence to $K_2$ all lie within a compact set, and contain a convergent 
subsubsequence $\mu_{21}, \mu_{22},\ldots$. We can extend this reasoning to $K_3$, $K_4$ and so on. 

Formally, denote by $p_n$ the projection from $\mfinite(E)$ to $\mfinite(K_n)$ and by $p_{mn}$ the projection from 
$\mfinite(K_m)$ to $\mfinite(K_n)$ for $m\geq n$. Then, we can rewrite the above as:
$$
p_1 \mu_{1n} \to \tilde \mu_1 \in \mfinite(K_1), \quad p_2 \mu_{2n} \to \tilde \mu_2 \in \mfinite(K_2), \quad, \ldots,
$$
where the convergence is with respect to the weak topology on the corresponding spaces. Now consider the diagonal 
sequence $\mu_{kk}$. It is clear from the above that 
$$
p_n\mu_{kk} \stackrel{k\rightarrow\infty}{\rightarrow} \tilde \mu_n \in \mfinite(K_n),
$$ 
for each $n$. A natural question to ask is whether there is a measure $\tilde \mu\in \mfinite(E)$ such that 
$\tilde \mu_n=p_n \tilde \mu$ for all $n$. The answer follows from a generalisation of Kolmogorov's Extension 
theorem by Yamasaki ~\cite[Proposition 2.1]{yamasaki75}; it is affirmative if the measures $\tilde \mu_n$ 
satisfy the consistency conditions $p_{mn} \tilde \mu_m=\tilde \mu_n$ for all $m>n$. It is straightforward 
to verify these.

We now show that the diagonal subsequence $\mu_{kk}$ converges weakly to the measure $\tilde \mu$ 
(whose existence we have just shown) in the weak topology on $\mfinite(E)$, and moreover that the limit 
$\tilde \mu$ is in $L_{(K_n,\epsilon_n)}$. We start with the latter. As $\tilde \mu$ is a finite measure on the 
Polish space $E$, it is regular; therefore, as $K_n$ are compact sets increasing to $E$, $\tilde \mu(K_n)$ 
increases to $\tilde \mu(E)$. Hence, for any $m\in \Ns$,
$$
\tilde \mu(K_m^c) = \lim_{n\to \infty} \tilde \mu(K_n) - \tilde \mu(K_m).
$$
Now, for any fixed $i>n>m$, $\tilde \mu_i$ is the restriction (or projection) of $\tilde \mu$ to the set $K_i$, 
and so 
$$
\tilde \mu(K_n) - \tilde \mu(K_m) = \tilde \mu_i(K_n) - \tilde \mu_i(K_m) \le \tilde \mu_i(K_m^c) \le \epsilon_m.
$$
The last inequality holds because $\tilde \mu_i$ is the weak limit of measures whose mass on $K_m^c$ is bounded 
by $\epsilon_m$, and $K_m^c$ is an open set. As this holds for each $n$, we conclude on taking limits that 
$\tilde \mu(K_m^c) \le \epsilon_m$. But $m$ was arbitrary, so $\tilde \mu \in L_{(K_n,\epsilon_n)}$.

Next, given $\delta>0$ and a bounded continuous function $g:E\to \Rs$, choose $\ell$ large enough that 
$\epsilon_{\ell} \| g \|_{\infty} < \delta$.  Next, pick $m\geq \ell$ large enough that 
$$
\Bigm| \int_{K_\ell} gd\mu_{\ell n}- \int_{K_\ell} gd\tilde \mu_{\ell} \Bigm| \le \delta \quad \forall \; n\ge m,
$$
which is possible since $\mu_{\ell n}$ converges weakly to $\tilde \mu_{\ell}$ as $n$ tends to infinity.
Now, $\mu_{n\cdot}$ is a subsequence of $\mu_{\ell\cdot}$ for $n\ge \ell$, so the above inequality also 
holds for $\int_{K_{\ell}} g(d\mu_{nn}-d\tilde \mu_{\ell})$ for all $n\ge m$. Thus, we can write
$$
\Bigm| \int_E gd\mu_{nn} - \int_E gd\tilde \mu \Bigm| \le \Bigm| \int_{K_{\ell}} g (d\mu_{nn} - d\tilde \mu_{\ell}) \Bigm| 
+ \Bigm| \int_{K_{\ell}} g (d\tilde \mu_{\ell} -d\tilde \mu) \Bigm| +2 \| g \|_{\infty} \epsilon_{\ell},
$$
as $\mu_{nn}(K_{\ell}^c)$ and $\tilde \mu(K_{\ell}^c)$ are both bounded above by $\epsilon_{\ell}$. We have just 
shown that the first integral above is smaller than $\delta$ in absolute value, for all $n\ge m$. The second integral is zero 
as $\tilde \mu_{\ell}$ is the restriction or projection of $\tilde \mu$ to $K_{\ell}$. The last term is bounded by $2\delta$ 
by the choice of $\ell$. Thus, we have shown that we can choose $m$ in such a way that
\begin{align*} \left|\int_E gd\mu_{nn} - \int_E gd\tilde \mu\right| \le 3\delta
\end{align*} 
for all $n\ge m$. As $g$ was an arbitrary bounded continuous function, this proves that $\mu_{nn}$ converges to $\tilde \mu$. This completes the proof that 
$L_{(K_n,\epsilon_n)}$ is compact.

For the converse, let $\mathcal{K}$ be compact in $\mfinite(E)$ equipped with the weak topology. As the map 
$\mu \mapsto \mu(E)$ is continuous (the indicator of $E$ is a bounded continuous function $E\to \Rs$), its 
supremum over $\mathcal{K}$ is attained. Denote the supremum by $\epsilon_0$. Then $\mu(E)=\mu(K_0^c) 
\le \epsilon_0$ for all $\mu \in \mathcal{K}$. Next, we invoke a generalisation of Prokhorov's theorem by 
Bogachev~\cite[Theorem 8.6.2]{bogachev07}), which states that the measures in a compact set are uniformly tight. 
In other words, given $\epsilon_1>0$, we can find a compact subset $K_1$ of $E$ such that $\mu(K_1^c) \le \epsilon_1$ 
for all $\mu \in \mathcal{K}$. Similarly, we can find compact $K_2$ such that $\mu(K_2^c) \le \epsilon_2$ for all 
$\mu \in \mathcal{K}$. Without loss of generality, we can assume that $K_1 \subseteq K_2$; otherwise, re-define $K_2$ 
as their union. Continuing in the same vein, we obtain a sequence $K_n$ of nested compact sets such that $\mu(K_n^c) 
\le \epsilon_n$ for all $n\in \Ns$, for all $\mu \in \mathcal{K}$. If their union is not equal to $E$, it can be extended countably 
to have this property, by the assumption that $E$ is $\sigma$-compact. Now, $\mathcal{K} \subseteq L_{(K_n,\epsilon_n)}$.
\hfill $\Box$

\section{Proof of LDP for Queue Occupancy and Departures} \label{sec:qldp}

The proof of Theorems~\ref{thm:qldp} and \ref{thm:dep_ldp} are presented in this section. We begin by recalling how the 
queue occupancy measure is related to the input to the queue. First, we represent the input to the $n^{\rm th}$ queue 
as a Cox process on $\Rs \times \Rs_+$ by marking each arrival with its service time; the resulting marked point process 
is a Cox process on $\Rs \times \Rs_+$ with stochastic intensity $\Lambda_n \otimes F$. Now, $Q_n(t)$ is equal to the 
number of points of this Cox process lying in the triangle 
\begin{align*}
A_t=\left\{(s,x)\in\mathbb{R}\times\mathbb{R}_+:s\leq t,x\geq t-s\right\}. 
\end{align*}
Furthermore, the queue length process $\left\{Q_n(t),t\in [a,b]\right\}$, is determined by the restriction of the above 
Cox process to the wedge
\begin{align*}
A_{[a,b]}:=\bigcup\limits_{t\in[a,b]}A_t, 
\end{align*}
as illustrated in Figure~\ref{truncatedwedge}. Next, for $u\leq s\leq t$, we will also need to define the truncated sets
\begin{align*}
A_t^u=\left\{(s,x)\in\mathbb{R}\times\mathbb{R}_+:u\leq s\leq t,x\geq t-s\right\}, \quad
A^u_{[s,t]}:=\bigcup\limits_{x\in[s,t]}A^u_x. 
\end{align*}
Finally, recall that we are interested in the occupancy measure $L_n$, which is defined as the random measure that is absolutely 
continuous with respect to Lebesgue measure, and has density $Q_n(\cdot)$.

\begin{figure}[htb]
\includegraphics[scale=0.3]{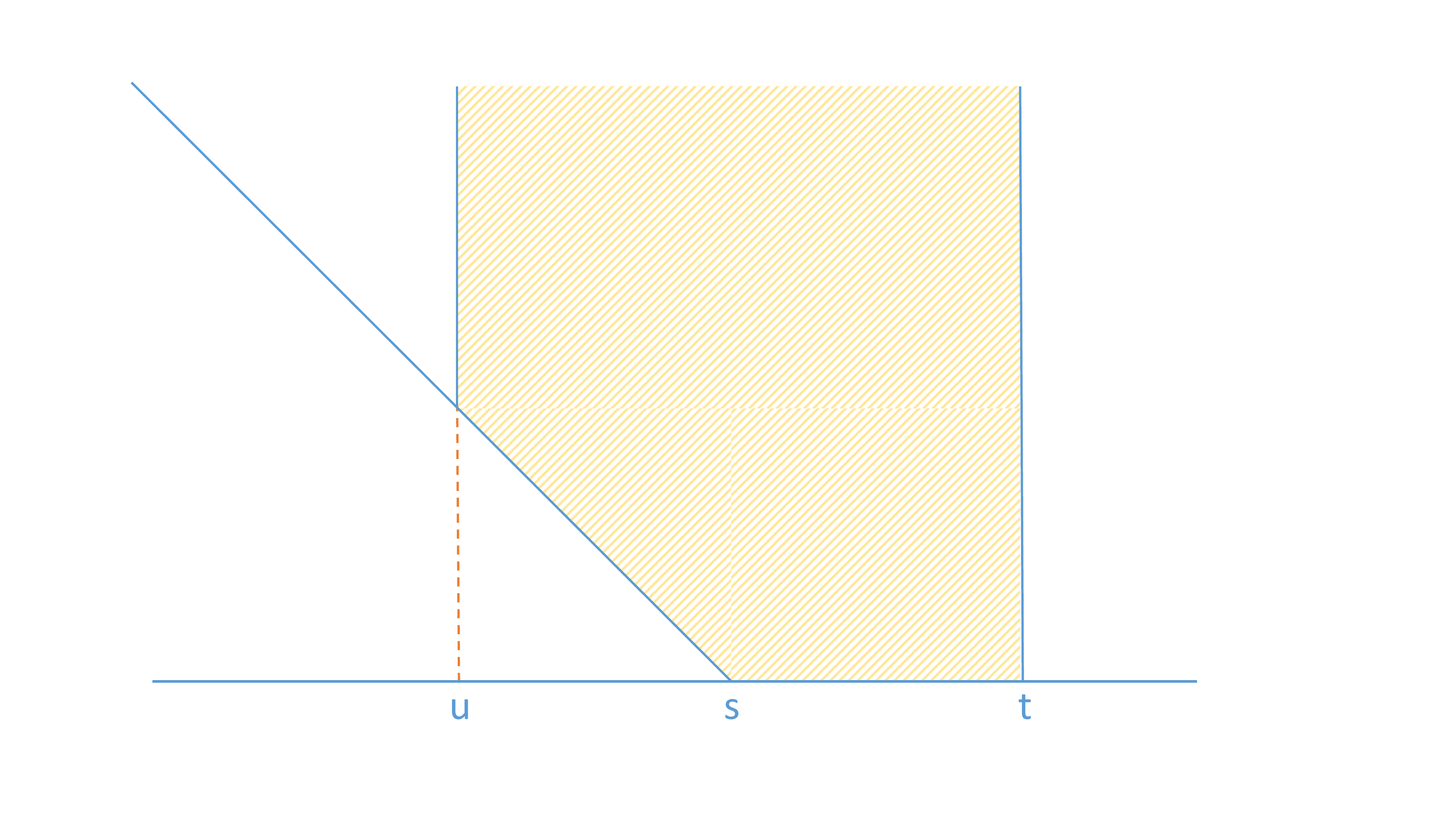}
\caption[Truncated Wedge]{The wedge $A_{[s,t]}$ and the truncated wedge $A_{[s,t]}^u$.}
\label{truncatedwedge}
\end{figure}

Our goal is to prove an LDP for $L_n$, restricted to an arbitrary interval $[a,b]$. We start by establishing an LDP for 
the scaled directing measures $\frac{\Lambda_n}{n}\otimes F$, restricted to a truncated wedge $A^u_{[a,b]}$, for 
arbitrary $u<a$; we define below a new topology, which we call the tempered topology, in which we establish this LDP. 
Then, using the projective limit approach described below, we extend this family of LDPs to an LDP on the full wedge 
$A_{[a,b]}$, in the projective limit topology. However, the queueing map is not continuous in this topology, so we need to 
strengthen the LDP to the weak topology on the full wedge. We do this by establishing exponential tightness of the measures 
$\frac{\Lambda_n}{n}\otimes F$ in the tempered topology on $A_{[a,b]}$. Next, we invoke Theorem~\ref{thm:cox_ldp} 
to deduce an LDP for the Cox process on $A_{[a,b]}$ with this intensity. Finally, we use continuity of the queueing map 
with respect to the weak topology, and the contraction principle, to obtain the LDP for $L_n$. Checking that $L_n$ also 
satisfies Assumptions [A1]-[A3] is fairly straightforward. The details of all these steps are presented below.

\begin{defin}\label{tempered}\normalfont
Let $u\le a<b \in \Rs$, and let $\mfinite(A^u_{[a,b]})$ denote the space of finite measures on the truncated wedge 
$A^u_{[a,b]}$ defined above. The tempered topology on this space is the weakest topology which makes the maps 
$\mu \mapsto \int fd\mu$ continuous for all bounded, continuous functions $f:A^u_{[a,b]} \to Rs$ which vanish at 
the boundary of $A^u_{[a,b]}$.

The tempered topology on $\mfinite(A_{[a,b]})$ is defined analogously.
\end{defin}

Notice that the tempered topology is weaker than the weak topology, as it is restricted to test functions that vanish 
at the boundary. We are now ready to state our first result.

\begin{lemma} \label{lem:ldp_trunc_wedge}
Fix $u\le a < b \in \Rs$ and consider the truncated wedge $A^u_{[a,b]}$. The sequence of random measures 
$\left. \frac{\Lambda_n}{n}\otimes F \right|_{A^u_{[a,b]}}$,  $n\in \Ns$, satisfy an LDP on $\mfinite(A^u_{[a,b]})$ 
equipped with the tempered topology, with good rate function 
$$
\eyecox^u_{[a,b]}(\mu)= \inf \left\{ \eyelamb_{[a,b]}(\lambda): \lambda \in \mfinite([a,b]),\; 
\mu= (\lambda \otimes F)\bigm|_{A^u_{[a,b]}} \right\} .
$$
\end{lemma}

\begin{proof} 
Define the map 
\begin{align*} T:\mfinite([u,b])\rightarrow\mfinite([u,b]\times\mathbb{R}_+) 
\end{align*}
by $T(\mu)=\mu\otimes F$. We first show that this map is continuous in the weak topology. As the weak topology is 
metrisable, we can check continuity along sequences. To this end, consider a sequence of finite measures $\mu_n$ 
on $[u,b]$ converging weakly to a finite measure $\mu$, and let $g:[u,b] \times \Rs_+ \to \Rs$ be bounded and 
continuous. Define $h:[u,b]\mapsto \Rs$ by $h(x)=\int_0^{\infty} g(x,y)dF(y)$. We have
$$
\int_{[u,b]\times \Rs_+} gd(T(\mu_n)) = \int_u^b \Bigl( \int_0^{\infty} g(x,y)dF(y) \Bigr) d\mu_n(x) = \int_u^b h(x)d\mu_n(x),
$$
where the first equality follows from Fubini's theorem.
If we can show that $h$ is continuous, then it will follow that $\int gd(T(\mu_n))$ converges to $\int gd(T(\mu))$, and, as $g$ was
an arbitrary bounded continuous function, that $T(\mu_n)$ converges weakly to $T(\mu)$, thus proving that $T$ is continuous.

Now, to show that $h$ is continuous, fix $\epsilon>0$ and $x_0\in \Rs$ such that $1-F(x_0)\le \epsilon$. Now $g$ is uniformly 
continuous on the compact set $[u,b]\times [0,x_0]$, so we can find $\delta>0$ such that $|g(x,z)-g(y,z)|<\epsilon$ provided 
$|x-y|<\delta$. It follows that 
\begin{eqnarray*}
&& |h(x)-h(y)| \\
&&\leq \int_0^{x_0} | g(x,z)-g(y,z) | dF(z) + \int_{x_0}^{\infty} |g(x,z)| dF(z) + \int_{x_0}^{\infty} |g(y,z)| dF(z) \\
&& \leq (1+2\| g\|_{\infty}) \epsilon.
\end{eqnarray*}
This proves the continuity of $h$, and consequently of $T$.

Next, let $S$ be the map that restricts finite measures on $[u,b]\times \Rs_+$ to the wedge $A^u_{[a,b]}$. Equip 
$\mfinite([u,b]\times \Rs_+)$ with the weak topology, and $\mfinite(A^u_{[a,b]})$ with the tempered topology. 
It is easy to see that $S$ is continuous. Indeed, let $\mu_n, n\in \Ns$ be a sequence of finite measures on 
$[u,b]\times \Rs_+$ converging weakly to a finite measure $\mu$ on $[u,b]\times \Rs_+$, and let $f$ be a 
bounded, continuous function on $A^u_{[a,b]}$, vanishing on its boundary. Extend it to a bounded, continuous 
function ${\hat f}:[u,b]\times \Rs_+ \to \Rs$ by defining ${\hat f} \equiv f$ on $A^u_{[a,b]}$ and ${\hat f} 
\equiv 0$ on the complement of $A^u_{[a,b]}$ in $[u,b]\times \Rs_+$. Then,
$$
\int_{A^u_{[a,b]}} fd(S(\mu_n)) = \int_{[u,b]\times \Rs_+} {\hat f}d\mu_n \rightarrow 
\int_{[u,b]\times \Rs_+} {\hat f}d\mu = \int_{A^u_{[a,b]}} fd(S(\mu)),
$$
where the convergence holds by the assumption that $\mu_n$ converge weakly to $\mu$. This proves that $S$ is continuous.
As $S$ and $T$ are both continuous, so is the composition $S\circ T$. The claim of the lemma now follows from the assumed 
LDP for $\frac{\Lambda_n}{n} \bigm|_{[u,b]}$ and the contraction principle~\cite[Theorem 4.2.1]{dembo98}.
\end{proof}

The family of LDPs on the truncated wedges $\{ A^u_{[a,b]}, u<a \}$ can be extended to an LDP on the full wedge 
$A_{[a,b]}$ using the Dawson-G\"artner theorem for projective limits~\cite[Theorem 4.6.1]{dembo98}. This yields an LDP 
in the projective limit topology, which is generated by bounded continuous functions supported on the truncated wedges 
$A^u_{[a,b]}$ and vanishing at their boundaries. In order to strengthen this LDP to the weak topology on $A_{[a,b]}$, 
we need to show exponential tightness of the measures $\frac{\Lambda_n}{n} \otimes F$ in the weak topology. The 
following lemma is a key ingredient in establishing this.

\begin{lemma} \label{cxorder} 
Suppose $X,X_1,X_2,...$ are identically distributed random variables with arbitrary joint distribution, and suppose 
$\alpha_i$, $i\in \Ns$ are non-negative coefficients whose sum is finite, and which we denote by $\alpha$. Then, 
$$
\sum\limits_{i=1}^{\infty} \alpha_i X_i \leq_{\rm cx} \alpha X, 
$$
where we write $Y\leq_{\rm cx} Z$ to denote that $Y$ is dominated by $Z$ in the convex stochastic order, i.e., 
$\E[\phi(Y)] \le \E[\phi(Z)]$ for all convex functions $\phi$ for which the expectations are defined, possibly infinite.
\end{lemma}

\begin{proof} 
By scaling the random variables, we assume $\alpha=1$ without loss of generality. By Jensen's inequality, the inequality
$$
\phi \Bigl( \sum_{i=1}^{\infty} \alpha_i X_i(\omega) \Bigr) \le \sum_{i=1}^{\infty} \alpha_i \phi(X_i(\omega)), 
$$
holds pointwise on the probability space $\Omega$. Taking expectations on both sides yields the result if we can 
interchange expectation and summation on the right. We can certainly do so (by Tonelli's theorem) 
if the functions $\phi$ are non-negative, and hence also if they are bounded below. Now, for any $c\in \Rs$, 
the function $\phi_c$ defined by $\phi_c(x) = \max \{c,\phi(x) \}$ is convex and bounded below, so we get
$$
\E \Bigl[ \phi_c \Bigl( \sum_{i=1}^{\infty} \alpha_i X_i \Bigr) \Bigr] \le 
\sum_{i=1}^{\infty} \alpha_i \E \Bigl[ \phi_c (X_i) \Bigr] = \Bigl( \sum_{i=1}^{\infty} \alpha_i \Bigr) \E [\phi_c(X) ], 
$$ 
as the $X_i$ are identically distributed with the same law as $X$. Since $\phi \leq \phi_c$, it follows that 
$$
\E \Bigl[ \phi \Bigl( \sum_{i=1}^{\infty} \alpha_i X_i \Bigr) \Bigr] \le \Bigl( \sum_{i=1}^{\infty} \alpha_i \Bigr) \E [\phi_c(X) ], 
$$ 
for all $c\in \Rs$. Letting $c$ decrease to $-\infty$ on the right now yields the claim of the lemma. This can be justified 
by splitting $\phi$ into its positive and negative parts, and using the Montone Convergence Theorem.
\end{proof}

We are now ready to show that the directing measures restricted to a wedge are exponentially tight in the weak topology. 

\begin{prop} \label{etlem} The sequence of random measures 
\begin{align*}
\left(\left.\left(\frac{\Lambda_n}{n}\otimes F\right)\right|_{A_{[a,b]}}\right)_{n\in\mathbb{N}} 
\end{align*}
is exponentially tight in the weak topology.
\end{prop}

\begin{proof} 
We have to show that for every $0<\alpha<\infty$, there is a compact set $\mathcal{K}_{\alpha} \subseteq 
\mfinite(A_{[a,b]})$ such that 
\begin{equation} \label{eq:exptight_bound1}
\limsup_{n\rightarrow\infty}\frac{1}{n}\log\ \Pb \left( \left.\left(\frac{\Lambda_n}{n}\otimes F\right)\right|_{A_{[a,b]}} \in \mathcal{K}_{\alpha}^c \right) <-\alpha.
\end{equation}
We will use the explicit construction of a weakly compact set of measures given in Proposition~\ref{compact_set_measures}. 
We seek a nested sequence of compact sets $K_1 \subseteq K_2 \subseteq \ldots \subseteq A_{[a,b]}$, whose union is the 
wedge $A_{[a,b]}$, and a sequence of positive constants $\epsilon_0 \ge \epsilon_1 \ge \ldots$ decreasing to zero, such that
\begin{equation} \label{eq:exptight_bound2}
\Pb \left( \Bigl( \frac{\Lambda_n}{n}\otimes F \Bigr) \Bigl( K_i^c \Bigr) >\epsilon_i\right) \leq e^{-n(i+1)\alpha} \quad 
\forall \; i\ge 0,
\end{equation}
where we define $K_0$ to be the empty set. If we can find such $K_i$ and $\epsilon_i$, then the weakly compact set of 
measures 
$$
\mathcal{K}_{\alpha} = \Bigl\{ \mu \in \mfinite(A_{[a,b]}): \mu(K_i^c) \le \epsilon_i \, \forall \, i\in \Ns \Bigr\},
$$ 
satisfies the inequality in (\ref{eq:exptight_bound1}), thus proving the proposition.

\begin{figure}[htb]
\includegraphics[scale=0.3]{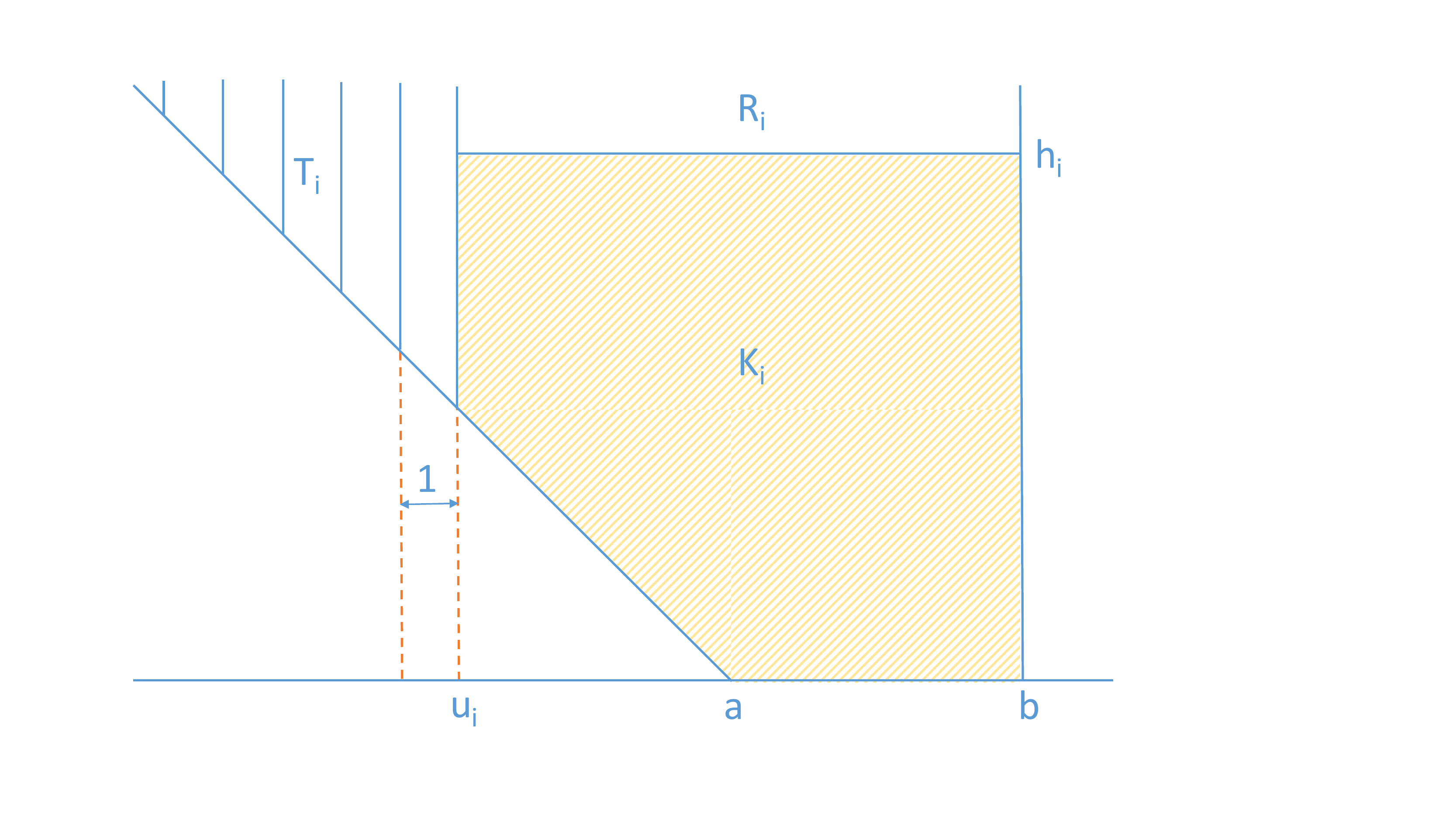}
\caption[Truncated Wedge With Strips]{The wedge $A_{[a,b]}$ split into a compact set $K_i$, infinite rectangle ${R_i}$ and infinite triangle ${T_i}$. The triangle is split into strips of unit width.}
\label{wedgewithstrips}
\end{figure}

Each of the compact sets $K_i$, $i\ge 1$, will be specified by two real numbers $u_i$ and $h_i$ as shown in 
Figure~\ref{wedgewithstrips}: 
\begin{align*} K_i = \{ [u_i,b] \times [0,h_i] \} \bigcap A_{[a,b]}. 
\end{align*}
We shall write $K_i^c$ to denote the complement of $K_i$ in $A_{[a,b]}$, and we decompose this set into a triangle 
\begin{align*} T_i = \left\{(s,x)\in\mathbb{R}\times\mathbb{R}_+:s\leq u_i,x\geq a-s\right\}, 
\end{align*}
and a rectangle 
\begin{align*} R_i = \left\{(s,x)\in\mathbb{R}\times\mathbb{R}_+:u_i\leq s\leq b,x\geq h_i\right\}; 
\end{align*}
see Figure~\ref{wedgewithstrips}. 
Thus, we have
\begin{equation} \label{eq:prob_complement}
\frac{1}{n} (\Lambda_n \otimes F)(K_i^c) = \frac{1}{n} (\Lambda_n \otimes F)(T_i) + \frac{1}{n} (\Lambda_n \otimes F)(R_i).
\end{equation}
Now, by the translation invariance of $\Lambda_n$, we have
$$
(\Lambda_n \otimes F)(T_i) \eqdistrib (\Lambda_n \otimes F)(T^{a-u_i}) \mbox{ and } 
(\Lambda_n \otimes F)(R_i) \eqdistrib \bigl( \Lambda_n \otimes F \bigr) \bigl( R^{h_i}_{b-u_i} \bigr)),
$$
where $\eqdistrib$ denotes equality in distribution, and the sets $T^{\ell}$ and $R^h_z$ are defined as
\begin{equation} \label{def_complement}
\begin{split}
T^{\ell} &= \{ (t,x)\in \Rs \times \Rs_+: t \le 0, t+x \ge \ell \} \\ 
R^{h}_{z} &= \{ (t,x)\in \Rs \times \Rs_+: t \in [0,z], x \ge h \}.
\end{split}
\end{equation}
Thus, we obtain from (\ref{eq:prob_complement}) that
\begin{eqnarray} 
\Pb \left( \Bigl( \frac{\Lambda_n}{n}\otimes F \Bigr) \Bigl( K_i^c \Bigr)>\epsilon_i\right) &\leq& 
\Pb \left( \bigl( \Lambda_n\otimes F \bigr) \bigl( T^{a-u_i} \bigr) >\frac{n\epsilon_i}{2} \right) \nonumber \\ 
&& + \Pb \left( \bigl( \Lambda_n \otimes F \bigr) \bigl( R^{h_i}_{b-u_i} \bigr) >\frac{n\epsilon_i}{2} \right). 
\label{eq:exptight_bound3}
\end{eqnarray}

We show in Lemma~\ref{lem:complement_bd} that, given $i\in \Ns$, $\epsilon_i>0$ and $\alpha>0$, we can choose 
$u_i$ to make $a-u_i$ sufficiently large that
$$
\Pb \Bigl( \bigl( \Lambda_n\otimes F \bigr) \bigl( T^{a-u_i} \bigr) > \frac{n\epsilon_i}{2} \Bigr) \le e^{-n(i+1)\alpha}, 
\quad \forall n\in \Ns;
$$
to see this, take $\epsilon=\epsilon_i/2$ and $\beta=(i+1)\alpha$ in the statement of the lemma. Next, by the same lemma, 
given $u_i$, and hence $b-u_i$, we can choose $h_i$ sufficiently large to ensure that 
$$
\Pb \Bigl( \bigl( \Lambda_n\otimes F \bigr) \bigl( R^{h_i}_{b-u_i} \bigr) > \frac{n\epsilon_i}{2}  \Bigr) \le e^{-n(i+1)\alpha}, 
\quad \forall n\in \Ns.
$$
Combining these two inequalities with (\ref{eq:exptight_bound3}), we conclude that for all $i\ge 1$, 
\begin{equation} \label{complement_prob_bd}
\Pb \Bigl( \bigl( \Lambda_n\otimes F \bigr) \bigl( K_i^c \bigr) > n\epsilon_i \Bigr) \le 2e^{-n(i+1)\alpha}, 
\quad \forall n\in \Ns,
\end{equation}
which is essentially the same as~(\ref{eq:exptight_bound2}). That leaves the case $i=0$.

The same argument does not work for $K_0$ as we cannot choose this set; $K_0$ is the empty set and 
$K_0^c = A_{[a,b]}$. Instead, we need to show that we can choose $\epsilon_0$ sufficiently large that 
\begin{equation} \label{total_mass_bd}
\Pb \Bigl( \bigl( \Lambda_n\otimes F \bigr) \bigl( A_{[a,b]} \bigr) > n\epsilon_0 \Bigr) \le e^{-n\alpha}, 
\quad \forall n\in \Ns.
\end{equation}
We first note that $A_{[a,b]} \subset T_0 \cup \{ [a-\ell, b]\times \Rs_+ \}$, where 
$$
T_0= \{ (t,x) \in \Rs \times \Rs_+ : t \le a-\ell, t+x\ge a \}.
$$
Hence 
\begin{align*} (\Lambda_n \otimes F)(A_{[a,b]}) \le (\Lambda_n \otimes F)(T_0) + \Lambda_n([a-\ell, b]).
\end{align*} 
Moreover, by translation invariance of $\Lambda_n$, we have 
\begin{align*} (\Lambda_n \otimes F)(T_0) \eqdistrib (\Lambda_n \otimes F)(T^{\ell}),
\end{align*} 
where $T^{\ell}$ is defined in~(\ref{def_complement}). Using Lemma~\ref{lem:complement_bd} below, we conclude that we 
can choose $\ell$ sufficiently large that 
\begin{equation} \label{t_ell_bd}
\Pb \Bigl( \bigl( \Lambda_n\otimes F \bigr) \bigl( T_0 \bigr) >n \Bigr) \le e^{-n\alpha}, \quad \forall n\in \Ns.
\end{equation}
We also see from the proof of Lemma~\ref{lem:complement_bd} that $\Lambda_n([a-\ell, b])$ is dominated, in the 
increasing convex order, by $\lceil \ell+b-a \rceil \Lambda_n([0,1])$; in particular,
$$
\E \left[ e^{\theta \Lambda_n([a-\ell, b])} \right] \leq \E \left[e^{\theta(\ell+1+b-a) \Lambda_n([0,1])} \right]
= \exp \Bigl( \psi_n \bigl( n\theta (\ell+1+b-a) \bigr) \Bigr),
$$
where $\psi_n$ is defined in Assumption [A3]. By [A3], for given $a,b,\ell$, $\psi_n(n\theta(\ell+1+b-a))/n$ is bounded, 
for $\theta$ in a neighbourhood of the origin, uniformly in $n$, i.e., there exist constants $\theta, \delta>0$ such that 
$\psi_n(n\theta) \le n\delta$ for all $n\in \Ns$. Consequently, by Markov's inequality,
$$
\Pb \left( \Lambda_n([a-\ell, b]) \geq n(\epsilon_0-1) \right) \leq e^{-n\theta (\epsilon_0-1)+n\delta}, \quad \forall n\in \Ns.
$$
Clearly, we can choose $\epsilon_0$ large enough to ensure that
$$
\Pb \left( \Lambda_n([a-\ell, b]) \geq n(\epsilon_0-1) \right) \leq e^{-n\alpha}, \quad \forall n\in \Ns.
$$
Combining the above equation with (\ref{t_ell_bd}), we see that the inequality in (\ref{total_mass_bd}) holds, up to a 
factor of two. This completes the proof that the inequality in (\ref{eq:exptight_bound2}) holds for all $i\ge 0$, up to a 
factor of two on the RHS. Now, using the union bound over $i$, we get 
$$
\Pb \left( \exists \, i\ge 0 : \Bigl( \frac{\Lambda_n}{n} \otimes F \Bigr) \Bigl( K_i^c \Bigr) > \epsilon_i \right) 
\leq \sum_{i=0}^{\infty} e^{-n(i+1)\alpha} \leq 2e^{-n\alpha},
$$
from which (\ref{eq:exptight_bound1}) is immediate, given the definition of $\mathcal{K}_{\alpha}$. This completes 
the proof of the proposition.
\end{proof}

\begin{lemma} \label{lem:complement_bd}
Let $\beta>0$ be a given constant. For $\ell, h, z>0$, let the triangle $T_{\ell}$ and the rectangle $R^h_z$  be defined 
as in (\ref{def_complement}). Then, we have the following:
\begin{enumerate}
\item
Given $\epsilon>0$, we can choose $\ell$ sufficiently large that
$$
\Pb \Bigl( \bigl( \Lambda_n\otimes F \bigr)\bigl( T^{\ell} \bigr) > n\epsilon \Bigr) \le e^{-n\beta}, \quad \forall n\in \Ns.
$$
\item
Given $z>0$ and $\epsilon>0$, we can choose $h$ sufficiently large that
$$
\Pb \Bigl( \bigl( \Lambda_n\otimes F \bigr) \bigl( R^h_{z} \bigr) > n\epsilon \Bigr) \le e^{-n\beta}, \quad \forall n\in \Ns.
$$
\end{enumerate}
\end{lemma}

\begin{proof} 
Fix an $\ell\in\Rs$. By splitting the triangle $T^{\ell}$ into vertical strips of unit width, we see that 
$$
\bigl( \Lambda_n \otimes F \bigr) \bigl( T^{\ell} \bigr) 
\le \sum_{k=0}^{\infty} \Lambda_n \bigl( [-k-1,-k] \bigr) \overline{F}(\ell+k).
$$
Now, by translation invariance of $\Lambda_n$, the random variables $\Lambda_n \bigl( [-k-1,-k] \bigr)$ are identically 
distributed for all $k$. Moreover, the sum of the coefficients $\overline{F}(\ell+k)$ can be bounded as follows:
$$
\sum_{k=0}^{\infty} \overline{F}(\ell+k) \leq c_{\ell} := \int_{\ell-1}^{\infty} \overline{F}(x)dx = \E [S\mathbf{1}(S\geq \ell-1)],
$$
where $S$ denotes a random variable with the distribution $F$ of the service time, and $\mathbf{1}(E)$ denotes the indicator 
of the event $E$. This last expectation is finite by the assumption that the service time has finite mean. Hence, invoking 
Lemma~\ref{cxorder}, we obtain that
$$
\bigl( \Lambda_n \otimes F \bigr) \bigl( T^{\ell} \bigr) \le_{icx} c_{\ell} \Lambda_n \bigl( [0,1] \bigr),
$$
where, for random variables $X$ and $Y$, we say that $X$ is dominated by $Y$ in the increasing convex order, written 
$X\le_{icx} Y$, if $\E[\phi(X)] \le \E[\phi(Y)]$ for all increasing convex functions $\phi$. Applying this bound to the 
increasing convex function $\phi(x)=e^{\theta x}$ for arbitrary $\theta>0$, and using Markov's inequality, we get, 
for any $\epsilon>0$,
$$
\Pb \Bigl( \bigl( \Lambda_n \otimes F \bigr) \bigl( T^{\ell} \bigr) \geq \frac{n\epsilon}{2} \Bigr) 
\leq e^{-n\theta \epsilon/2} \E \left[ e^{\theta c_{\ell} \Lambda_n([0,1])} \right] 
= \exp \Bigl( -\frac{n\theta \epsilon}{2} + \psi_n \bigl( n\theta c_{\ell} \bigr) \Bigr),
$$
where the function $\psi_n$ was defined in Assumption [A3]. As $\theta>0$ is arbitrary, it is convenient to rewrite the 
above inequality, replacing $\theta$ by $\theta/c_{\ell}$, as
\begin{equation} \label{eq:triangle_prob_bd}
\log \Pb \Bigl( \bigl( \Lambda_n \otimes F \bigr) \bigl( T^{\ell} \bigr) \geq \frac{n\epsilon}{2} \Bigr) \leq
-\frac{n\theta \epsilon}{2c_{\ell}} + \psi_n (n\theta), \mbox{ where } c_{\ell} = \E [S\mathbf{1}(S\geq \ell-1)].
\end{equation}
Now, by Assumption [A3], there exist positive constants $\delta$ and $\theta$ such that $\psi_n(n\theta) \le n\delta$, 
uniformly in $n$. Morever, as $\E[S]$ is finite by Assumption [A4], it follows that $c_{\ell}$ tends to zero as $\ell$ tends 
to infinity. Hence, we see from (\ref{eq:triangle_prob_bd}) that, given $i\in \Ns$ and $\beta, \epsilon>0$, we can 
choose $\ell$ sufficiently large, and consequently $c_{\ell}$ sufficiently small, to ensure that
\begin{equation} \label{eq:triangle_prob_bd2}
\Pb \left( \bigl( \Lambda_n \otimes F \bigr) \bigl( T^{\ell} \bigr) \geq n\epsilon \right) \leq e^{-n\beta} \quad \forall \; n\in \Ns.
\end{equation}
This completes the proof of the first claim of the lemma.

The proof of the second claim is very similar. We show that 
$$
\bigl( \Lambda_n \otimes F \bigr) \bigl( R^h_{b-a} \bigr) \le_{icx} \lceil b-a \rceil \overline{F}(h) \Lambda_n ([0,1]),
$$
and apply Markov's inequality to the exponential of the random variable on the RHS. The details are omitted. 
\end{proof}

We now have all the ingredients required to establish an LDP for the scaled intensity measures $(\Lambda_n \otimes F)/n$, 
on the wedge $A_{[a,b]}$.
\begin{prop} \label{intldp} 
Suppose that $\Lambda_n, n\in \Ns$ is a sequence of random measures satisfying Assumptions [A1]-[A3] and $F$ satisfies [A4]. 
Fix an interval $[a,b] \subset \Rs$. The sequence of random measures 
$\left.\left(\frac{\Lambda_n}{n}\otimes F\right)\right|_{A_{[a,b]}}$, $n\in \Ns$, satisfy an LDP on $\mfinite(A_{[a,b]})$ 
equipped with the weak topology, with good rate function 
$$
\eyecox_{[a,b]}(\nu) = \sup_{u\le a} \eyecox^u_{[a,b]} \bigl( \nu \bigm|_{A^u_{[a,b]}} \bigr), \quad \nu \in \mfinite([a,b]).
$$
\end{prop}

\begin{proof} 
We will use the Dawson-G\"artner theorem~\cite[Theorem 4.6.1]{dembo98} for projective limits. Letting 
\begin{align*} J:=\left\{ A^u_{[a,b]}: u\in(-\infty,a)\right\}, 
\end{align*}
it is clear that the collection $(J,\subseteq)$ of truncated wedges $A^u_{[a,b]}$ equipped with set inclusion is totally ordered, 
and hence also right-filtering. The set is indexed by $u$, and we will use $u$ to denote the element $A^u_{[a,b]}$, to simplify 
notation. Denote by $\mathcal{Y}_u$ the space $\mfinite(A^u_{[a,b]})$ of finite measures on $A^u_{[a,b]}$, equipped with 
the tempered topology.

If $t\le u$, i.e., $A^u_{[a,b]} \subseteq A^t_{[a,b]}$ (note that the order in the projective system reverses 
inequalities from the order on the real line), define the projection $p_{ut}: \mathcal{Y}_t \to \mathcal{Y}_u$ by the 
restriction of a measure on $A^t_{[a,b]}$ to the subset $A^u_{[a,b]}$. It is clear that this map is continuous in 
the tempered topology, since any bounded, continuous function on $A^u_{[a,b]}$, vanishing on its boundary, can be 
extended to a bounded, continuous function on $A^t_{[a,b]}$, vanishing on its boundary, by setting it to zero outside 
$A^u_{[a,b]}$. Moreover, the projections satisfy the consistency condition $p_{us}=p_{ut}\circ p_{ts}$ for $s\le t\le u$. 
Thus, $(\mathcal{Y}_u, p_{ut})_{t\le u}$ constitute a projective system. We can identify $\mfinite(A_{[a,b]})$ 
with the projective limit, with canonical projections 
\begin{align*} p_u: \mfinite(A_{[a,b]})\to \mfinite(A^u_{[a,b]}) 
\end{align*}
defined as the restriction of a measure from the full wedge $A_{[a,b]}$ to its truncation $A^u_{[a,b]}$. These are 
clearly continuous in the tempered topology, by the same argument as above. 

Now, by Lemma~\ref{lem:ldp_trunc_wedge}, the projections 
$$
\left. \left( \frac{\Lambda_n}{n}\otimes F \right) \right|_{A^u_{[a,b]}} = p_u \left( \left.  \left( \frac{\Lambda_n}{n}\otimes F 
\right)  \right|_{A_{[a,b]}} \right), \; n\in \Ns,
$$
satisfy an LDP for each $u\in (\infty,a)$, with rate function $\eyecox^u_{[a,b]}$. Hence, by the Dawson-G\"artner 
theorem, the sequence of measures $\left. \left( \frac{\Lambda_n}{n}\otimes F \right) \right|_{A_{[a,b]}}$, $n\in \Ns$, 
satisfies an LDP in the projective limit topology, with good rate function
$$
\eyecox_{[a,b]}(\nu) = \sup_{u\le a} \eyecox^u_{[a,b]} \bigl( \nu \bigm|_{A^u_{[a,b]}} \bigr), \quad \nu \in \mfinite([a,b]).
$$
Moreover, by Proposition~\ref{etlem}, the measures $\left. \left( \frac{\Lambda_n}{n}\otimes F \right) \right|_{A_{[a,b]}}$ 
are exponentially tight in the weak topology on $\mfinite(A_{[a,b]})$. Hence, by \cite[Corollary 4.2.6]{dembo98}, we obtain that 
the LDP holds in the weak topology. Exponential tightness also implies goodness of the rate function~\cite[Lemma 1.2.18]{dembo98}.
\end{proof}

Next, we show the continuity of the queueing map, which is the prelude to obtaining the LDP for the queue occupancy measure. 
For a measure $\nu \in \mfinite(A_{[a,b]})$, and $t\in [a,b]$, we define $Q^{\nu}(t)=\nu(A_t)$, where we recall that 
$A_t=A_{[t,t]}$ is the set 
\begin{align*} \{ (s,x)\in \Rs \times \Rs_+: s\le t, s+x\ge t \}. 
\end{align*}
The interpretation is that, if $\nu$ is a counting measure representing the marked arrival process into an infinite-server queue, 
where each arrival is marked with its service time, then $Q^{\nu}(t)$ denotes the number of customers in the queue at 
time $t$. Let $L(\nu)$ denote the measure on $[a,b]$ which is absolutely continuous with respect to Lebesgue measure, 
and has density $Q^{\nu}(\cdot)$; let $L$ denote the map from $\mfinite(A_{[a,b]})$ to $\mfinite([a,b])$ which takes 
$\nu$ to $L(\nu)$.

We want an explicit characterisation of the map $L$. We will describe $L(\nu)$ through its action on the dual space 
$C_b([a,b])$ of bounded, continuous functions on $[a,b]$, i.e., by specifying $\int_a^b g(t)dL(\nu)(t)$ for all 
$g\in C_b([a,b])$. By the Riesz representation theorem, $L(\nu)$ is uniquely determined by these integrals. From the 
description above, we have
\begin{eqnarray} 
\int_a^b g(t)dL(\nu)(t) &=& \int_a^b g(t) Q^{\nu}(t) dt =\int_{t=a}^b g(t) \nu(A_t) dt \nonumber \\
&=& \int_{A_{[a,b]}} \Bigl( \int_{\max\{a,s\}}^{\min\{s+x,b\}} g(t)dt \Bigr) \nu(ds\times dx). \label{eq:qdual}
\end{eqnarray}
The last equality is obtained by interchanging the order of integration, noting that an area element at $ds\times dx$ 
contributes to $\nu(A_t)$ for each $t$ between $\max \{a,s\}$ and $\min \{s+x,b\}$.

\begin{lemma}\label{weaklycontmap} The map $L:\mfinite(A_{[a,b]}) \to \mfinite{([a,b])}$, defined by (\ref{eq:qdual}) 
via the Riesz representation theorem, is continuous with respect to the weak topology on each of these sets.
\end{lemma}

\begin{proof}
The weak topology on the space of finite measures on a Polish space is metrisable~\cite{varadarajan58}, so we can check 
continuity of $L$ along sequences. Suppose $\nu_n, n\in \Ns$ converge to $\nu$ in the weak topology on $\mfinite(A_{[a,b]})$. 
Let $g:[a,b]\to \Rs$ be a bounded, continuous function. We have by (\ref{eq:qdual}) that 
\begin{eqnarray}
\int_a^b g(t)dL(\nu_n)(t) &=& \int_{A_{[a,b]}} h(s,x) \nu_n(ds\times dx), \nonumber\\
\mbox{where} && h(s,x) = \int_{\max\{a,s\}}^{\min\{s+x,b\}} g(t)dt, \label{eq:qmeasure}
\end{eqnarray}
where the last integral is defined to be zero if the upper limit of integration is smaller than the lower limit. (In other words, 
the domain of integration should be understood to be empty in this case, rather than treating it as a signed integral with 
limits reversed.)

It is clear that the the function $h:A_{[a,b]} \to \Rs$ is bounded and continuous.
Hence, it follows from the assumed convergence of $\nu_n$ to $\nu$ in the weak topology that the RHS in (\ref{eq:qmeasure}) 
converges to 
\begin{align*} \int_{A_{[a,b]}} h(s,x) \nu(ds\times dx). 
\end{align*}
This completes the proof of the lemma.
\end{proof}

We are now ready to prove the main result. 

\noindent {\bf Proof of Theorem~\ref{thm:qldp}.}
Let $\Phi_n$ denote the Cox process of arrivals into the $n^{\rm th}$ queue, marked with their service times.
Fix $[a,b]\subset \Rs$. By Proposition~\ref{intldp}, the sequence of measures 
$\left. \left( \frac{\Lambda_n}{n}\otimes F \right) \right|_{A_{[a,b]}}$, 
satisfy an LDP on $\mfinite(A_{[a,b]})$ equipped with the weak topology, with good rate function $\eyecox_{[a,b]}$ 
given therein. Hence, by Theorem~\ref{thm:cox_ldp}, the sequence of Cox point measures 
$\left. \frac{\Phi_n}{n} \right|_{A_{[a,b]}}$ also satisfies an LDP on $\mfinite(A_{[a,b]})$ equipped with the weak topology, 
with good rate function $\eyecoxpp_{[a,b]}$ given by
\begin{equation} \label{eq:coxppldp_zero}
\eyecoxpp_{[a,b]}(\boldsymbol{0}) = \inf_{\lambda} \left\{ \eyecox_{[a,b]}(\lambda) + \lambda \left( A_{[a,b]} \right)\right\}, 
\end{equation}
where $\boldsymbol{0}$ denotes the zero measure, whereas, for $\mu \not\equiv \boldsymbol{0}$,
\begin{eqnarray} 
\eyecoxpp_{[a,b]}(\mu) &=& \inf_{\lambda} \Bigl\{ \eyecox_{[a,b]}(\lambda) + \eyep \left( \mu(A_{[a,b]}), \lambda(A_{[a,b]}) 
\right) 
\nonumber \\
&& \quad\quad + \mu(A_{[a,b]}) H\Bigl( \frac{\mu}{\mu(A_{[a,b]})} \Bigm| \frac{\lambda}{\lambda(A_{[a,b]})} \Bigr)
\Bigr\}, \label{eq:coxppldp}
\end{eqnarray}
where $H$ and $\eyep$ are defined in the statements of Theorem~\ref{baxterthm} and Lemma~\ref{poisson_ldp} respectively.

Now, the queue occupancy measures $L_n$ are given by $L_n/n = L(\Phi_n/n)$, where the map $L$ is defined by 
(\ref{eq:qdual}), and is linear and weakly continuous. Hence, by the contraction principle~\cite[Theorem 4.2.1]{dembo98}, 
the sequence of measures $L_n/n$ satisfies an LDP on $\mfinite([a,b])$ equipped with the weak topology, with good 
rate function
\begin{equation} \label{eq:qocc_ldp}
\eyeq_{[a,b]}(\nu) = \inf \left\{ \eyecoxpp_{[a,b]}(\mu): L(\mu)=\nu \right\},
\end{equation}
where the infimum of an empty set is defined to be $+\infty$. Thus, the sequence $L_n$ satisfies Assumption [A2]. 
The measures $L_n$ inherit translation invariance from $\Lambda_n$ via $\Lambda_n \otimes F$ and $\Phi_n$, while 
finiteness of the mean follows easily from that of $\lambda$ (the mean arrival intensity) and of the service time distribution. 
Thus, [A1] is verified. It remains to check [A3].

Observe that, analogous to (\ref{eq:qmeasure}), we have 
\begin{eqnarray*}
L_n([0,1]) &=& (L(\Phi_n))([0,1]) \\
&=& \int_{(s,x)\in A_{[0,1]}} \bigl( \min \{s+x,1\} - \max \{s,0\} \bigr) \Phi_n(ds\times dx) \\
&\leq& \Phi_n(A_{[0,1]}).
\end{eqnarray*}
But, conditional on $\Lambda_n \equiv \boldsymbol{\lambda}$, $\Phi_n([0,1])$ is a Poisson random variable with mean 
$(\boldsymbol{\lambda}\otimes F)(A_{[0,1]})$. Hence, we have for $\theta \geq 0$ that
$$
\E \left[ e^{\theta L_n([0,1])} \right] \leq \E \left[ \exp \Bigl( \bigl( e^{\theta}-1 \bigr) \bigl( \Lambda_n \otimes F \bigr)
\bigl(A_{[0,1]} \bigr) \Bigr) \right].
$$
Moreover, it can be shown by splitting $A_{[0,1]}$ into vertical strips of unit width and invoking Lemma~\ref{cxorder}, as in the 
proof of Lemma~\ref{lem:complement_bd}, that 
$$
(\Lambda_n \otimes F)(A_{[0,1]}) \le_{icx} (1+\E[S])\Lambda_n([0,1]),
$$
where $\E[S]$ denotes the mean service time, and is finite by Assumption [A4]. Hence, we obtain for $\theta \geq 0$ that
$$
\E \left[ e^{\theta L_n([0,1])} \right] \leq \E \left[ \exp \Bigl( \bigl( e^{\theta}-1 \bigr) \bigl( 1+\E[S] \bigr)
\bigl( \Lambda_n([0,1]) \bigr) \Bigr) \right].
$$
By Assumption [A3], there is a neighbourhood of 0 on which 
\begin{align*} \frac{\psi_n(n\eta)}{n} =\frac{1}{n}\log \E \left[e^{\eta \Lambda_n(0,1)}\right]
\end{align*} 
is bounded, uniformly in $n$. Setting $\eta=(e^{\theta}-1)(1+\E[S])$, we obtain uniform boundedness of 
\begin{align*} \frac{1}{n}\log \E\left[e^{\theta L_n([0,1])}\right] 
\end{align*}
for $\theta \geq 0$ sufficiently small, uniformly in $n$. Boundedness is automatic for $\theta<0$ as the random variables 
$L_n([0,1])$ are non-negative. Thus, the sequence of measures $L_n$ satisfy [A3] as well. 
This completes the proof of the theorem.
\hfill $\Box$

Having established the LDP for the queue occupancy measure, we now turn to the empirical measure of the departure process 
from the infinite-server queue, which can be expressed as a function of the marked arrival process, where the marks specify the 
service times. Fix a compact interval $[a,b]\subset \Rs$, and let $D$ denote the function which maps the marked arrival process 
on $A_{[a,b]}$, to the departure process measure on $[a,b]$, as described in (\ref{eq:dep_measure}). We will formally define 
$D$ via the Riesz representation theorem, by specifying, for each $\nu \in \mfinite(A_{[a,b]})$, the integral with respect to 
$D(\nu)$ of arbitrary bounded, continuous functions on $[a,b]$. Let $g \in C_b([a,b])$ be one such function. We define the 
function $h_g$ on $A_{[a,b]}$ by setting 
\begin{equation} \label{eq:hdef}
h_g(s,x) = \begin{cases}
g(s+x), & (s,x) \in {\rm cl}(A_{[a,b]} \backslash A_b), \\
0, & \mbox{ otherwise,}
\end{cases}
\end{equation}
and define the map $\nu \mapsto D(\nu)$ by setting
\begin{equation} \label{eq:dep_riesz}
\int_{[a,b]} gd(D(\nu)) = \int_{A_{[a,b]}} hd\nu, \quad \forall \; g\in C_b([a,b]).
\end{equation}
It is clear from (\ref{eq:dep_measure}) that 
$$
\int_{[a,b]} gd\Psi_n = \int_{A_{[a,b]}} hd\hat\Phi_n, \quad \forall \; g\in C_b([a,b]),
$$
i.e., $\Psi_n = D(\hat \Phi_n)$. We will show that $D(\cdot)$ is continuous in a suitable topology, and use this to establish the 
desired LDP for $(\Psi_n, n\in \Ns)$.

\noindent {\bf Proof of Theorem~\ref{thm:dep_ldp}.} 
We begin by showing that the map $D:\mfinite(A_{[a,b]}) \to \mfinite([a,b])$ defined by (\ref{eq:dep_riesz}) is continuous, 
when $\mfinite(A_{[a,b]})$ is equipped with the weak topology, and $\mfinite([a,b])$ with the tempered topology. We can 
check continuity using sequences, as the weak topology on $\mfinite(A_{[a,b]})$ is metrisable~\cite{varadarajan58}. Consider 
a sequence of finite measures $\nu_n$ on $A_{[a,b]}$, converging weakly to a finite measure $\nu$. Let $g$ be a bounded, 
continuous function on $[a,b]$, vanishing at its end-points, $a$ and $b$. Then, it is clear that the function $h_g$ defined in 
(\ref{eq:hdef}) is bounded and continuous on $A_{[a,b]}$. Therefore, $\int h_g d\nu_n$ converges to $\int h_d d\nu$, where 
the integrals are over $A_{[a,b]}$. Hence, by (\ref{eq:dep_riesz}), $\int gdD(\nu_n)$ converges to $\int gdD(\nu)$. It follows 
that $\nu_n$ converges to $\nu$ in the tempered topology.

It was shown in the proof of Theorem~\ref{thm:qldp} that $\hat\Phi_n/n$ (which was denoted $\Phi_n/n$ there!) satisfy 
the LDP on $\mfinite(A_{[a,b]})$ equipped with the weak topology. Since the map $D(\cdot)$ is continuous, it follows by 
the contraction principle~\cite[Theorem 4.2.1]{dembo98} that $\Psi_n/n$ satisfy the LDP on $\mfinite([a,b])$ equipped with 
the tempered topology, with a good rate function $K_{[a,b]}(\cdot)$, which can be expressed as the solution of a minimisation 
problem. 

It remains to strengthen this LDP to the weak topology on $\mfinite([a,b])$. We do this by showing that the sequence of 
random variables $\Psi_n$, $n\in \Ns$, is exponentially tight in the weak topology. In order to show this, fix $\alpha>0$, 
arbitrarily large. We need to find a weakly compact subset $K$ of $\mfinite([a,b])$ such that 
$$
\limsup_{n\to \infty} \frac{1}{n} \log \Pb \Bigl( \frac{\Psi_n}{n} \in K^c \Bigr) \leq -\alpha,
$$
where $K^c$ denotes the complement of $K$. Fix $\gamma>0$ sufficiently large, and take 
$$
K(\gamma)= \{ \nu \in \mfinite([a,b]): \nu([a,b]) \leq \gamma \}.
$$ 
Then $K(\gamma)$ is compact in the weak topology, as noted in the proof of Proposition~\ref{compact_set_measures}. 
Moreover, 
$$
\Psi_n([a,b]) = \hat\Phi_n( {\rm cl} (A_{[a,b]} \backslash A_b) ) \leq \hat\Phi_n(A_{[a,b]}), 
$$
and so,
$$
\limsup_{n\to \infty} \frac{1}{n} \log \Pb \Bigl( \frac{\Psi_n}{n} \in K(\gamma)^c \Bigr) \leq 
\limsup_{n\to \infty} \frac{1}{n} \log \Pb \Bigl( \frac{\hat\Phi_n}{n} \bigl( A_{[a,b]} \bigr) > \gamma \Bigr).
$$
By the goodness of the rate function governing the LDP of $\hat\Phi_n/n$, the last term tends to $-\infty$ as $\gamma$ 
tends to infinity. Hence, we can choose $\gamma$ large enough to make it smaller than $-\alpha$, as required.

Since $(\Psi_n/n, n\in \Ns)$ satisfy the LDP on $\mfinite([a,b])$ equipped with the tempered topology, and are exponentially 
tight in the weak topology, it follows by~\cite[Corollary 4.2.6]{dembo98} that the LDP also holds in the weak topology, and 
by~\cite[Lemma 1.2.18]{dembo98} that the rate function is good.

It remains to show that $\Psi_n/n$ satisfy Assumptions [A1] and [A3]. The proof is very similar to the corresponding part of the 
proof of Theorem~\ref{thm:qldp}. Translation invariance is inherited from $\hat\Phi_n/n$, and finiteness of the mean intensity 
is also easy to prove using the same property for $\Phi_n/n$ and $F$. To prove [A3], we use the fact that $\Psi_n([a,b])$ is 
dominated by $\hat\Phi_n(A_{[a,b]})$. We omit the details, which are identical to the proof of Theorem~\ref{thm:qldp}.
\hfill $\Box$

{\bf Acknowledgements} The second author learnt of the biological problem that motivated this work at a workshop at the 
Mathematical Biosciences Institute, Ohio, in September 2011. He would like to thank Tom Kurtz and the MBI for the invitation.

\end{document}